\documentclass[12pt]{amsart}

\usepackage[a4paper,left=20mm,right=20mm,top=25mm,headheight=5mm,headsep=5mm,bottom=25mm,footskip=10mm]{geometry}
\usepackage[T1]{fontenc} 
\usepackage[active]{srcltx}
\usepackage{amsthm,amsmath,amssymb,amsfonts,amscd}
 \usepackage{bbm}
\usepackage{hyperref}
\hypersetup{colorlinks=true,allcolors=blue}
\usepackage{graphicx}
\usepackage{enumerate}
\usepackage{enumitem}
\usepackage[dvipsnames]{xcolor}
\usepackage{physics}
\usepackage{pmboxdraw}
\usepackage{listings}
\usepackage{tikz}
\usepackage{mathtools}
\usepackage{float}
\usepackage{xifthen}
\usepackage{bm}
\usepackage{subfig}
\usepackage{inconsolata}
\usepackage{centernot}
\usepackage[framemethod=tikz]{mdframed}

\usepackage{subfig}

\newtheorem{theorem}{Theorem}[section]

\numberwithin{equation}{subsection}
\numberwithin{figure}{subsection}
\numberwithin{theorem}{subsection}

\newtheorem{proposition}[theorem]{Proposition}
\newtheorem{lemma}[theorem]{Lemma}
\newtheorem{corollary}[theorem]{Corollary}
\newtheorem{conjecture}[theorem]{Conjecture}

\theoremstyle{definition}
\newtheorem{definition}[theorem]{Definition}

\theoremstyle{remark}

\newtheorem{remark}[theorem]{Remark}

\newenvironment{equations}
    {\begin{equation}\begin{aligned}}
    {\end{aligned}\end{equation}}

\theoremstyle{plain}
\newmdtheoremenv[%
   backgroundcolor=cyan!10,%
   outerlinecolor=black,%
   innerbottommargin = \topskip,%
   roundcorner=4]
{QUESTION}{Question}

\newcommand{\N}{\mathbb{N}}
\newcommand{\Z}{\mathbb{Z}}
\newcommand{\Q}{\mathbb{Q}}
\newcommand{\R}{\mathbb{R}}
\newcommand{\C}{\mathbb{C}}

\newcommand{\iu}{\mathrm{i}\mkern1mu}
\newcommand*{\Chi}{\mbox{\Large$\chi$}}
\newcommand{\Fourier}[1]{\mathcal{F}\mathopen{}\left\{#1\right\}\mathclose{}}
\DeclareMathOperator{\lct}{lct}

\newcommand{\mathbbold}[1]{\mbox{\fontencoding{U}\fontfamily{bbold}\fontseries{m}\fontshape{n}\selectfont#1}}
\newcommand{\bbI}{\mathbbold{1}}
\newcommand{\I}[1]{\bbI_{\mbox{$#1$}}}

\DeclarePairedDelimiter\floor{\lfloor}{\rfloor}
\DeclarePairedDelimiter\fract{\{}{\}}
\DeclarePairedDelimiter\ang{\langle}{\rangle}
\definecolor{blue-violet}{rgb}{0.54, 0.17, 0.89}
\definecolor{Blue}{rgb}{0.01, 0.28, 1.0}
\definecolor{gGreen}{rgb}{0.2, 0.8, 0.2}
\definecolor{Green}{rgb}{0.04, 0.85, 0.32}

\begin{document}

\title[Limit distribution of Hodge spectral exponents]{Limit distribution of Hodge spectral exponents of irreducible plane curve singularities}

\author[M. Alberich-Carrami\~nana]{Maria Alberich-Carrami\~nana}
\address{
    Departament de Matem\`atiques  and  Institut de Matem\`atiques de la UPC-BarcelonaTech (IMTech)\\  Universitat Polit\`ecnica de Catalunya \\ Av.~Diagonal 647, Barcelona 08028; and Institut de Rob\`otica i Inform\`atica Industrial\\ CSIC-UPC \\ Llorens i Artigues 4-6, Barcelona 08028, Spain.
    } 
\email{Maria.Alberich@upc.edu}

\author[J. \`Alvarez Montaner]{Josep \`Alvarez Montaner}
\address{
    Departament de Matem\`atiques  and  Institut de Matem\`atiques de la UPC-BarcelonaTech (IMTech)\\  Universitat Polit\`ecnica de Catalunya \\ Av.~Diagonal 647, Barcelona 08028; and Centre de Recerca Matem\`atica (CRM), Bellaterra, Barcelona 08193.
    } 
\email{Josep.Alvarez@upc.edu}

\author[R. G\'omez-L\'opez]{Roger G\'omez-L\'opez}
\address{
    Departament de Matem\`atiques \\  Universitat Polit\`ecnica de Catalunya \\ Av.~Diagonal 647, Barcelona 08028
    }
\email{Roger.Gomez.Lopez@upc.edu}

\thanks{
    All three authors are partially supported by grant  PID2019-103849GB-I00 (MCIN/AEI/10.13039/501100011033) and AGAUR grant 2021 SGR 00603. JAM is also supported by Spanish State Research Agency, through the Severo Ochoa and Mar\'ia de Maeztu Program for Centers and Units of Excellence in R$\&$D (project CEX2020-001084-M). The third author gratefully acknowledges Secretaria d'Universitats i Recerca del Departament d’Empresa i Coneixement de la Generalitat de Catalunya and Fons Social Europeu Plus for the financial support of his FI Joan Oró predoctoral grant.
    }


\begin{abstract}
We study the distribution of the Hodge spectral exponents of an irreducible plane curve by comparing it with a continuous distribution. We provide a closed formula for this difference in terms of numerical invariants of the curve. We characterize those families of irreducible plane curves whose limit distribution of Hodge spectral exponents is the continuous distribution and we provide intervals of dominating values. 
\end{abstract}

\maketitle

\section{Introduction}\label{section:background:Hodge}

Let $f\colon (\C^{n+1},0)\rightarrow (\C,0)$ be a germ of a holomorphic function (or equivalently a convergent power series $f\in\C\{x_0,\dots,x_n\}$) with an isolated singularity at the origin. 
Using the canonical mixed Hodge structure of the cohomology groups of the Milnor fiber of $f$, Steenbrink \cite{steenbrink1977mixed} defined the {\it Hodge spectrum} of $f$ as the generating function 
$$\operatorname{Sp}_f(T)=\sum_{i=1}^\mu T^{\alpha_i},$$
where  $\mu = \dim_\C \frac{\C\{x_0,\dots,x_n\}}{\left(\frac{df}{dx_0},\dots,\frac{df}{dx_n}\right)}$ is the \textit{Milnor number} and the positive rational numbers 
\begin{equation*}
        0< \alpha_1 \leqslant \dots \leqslant \alpha_\mu < n+1
\end{equation*}
form a discrete set of invariants of the singularity $f$ called \textit{Hodge spectral exponents}  (or simply \textit{spectral numbers}).  
They are symmetric with respect to $(n+1)/2$, i.e. for every $j=1,\dots,\mu$, we have $\alpha_{\mu+1-j} = (n+1) - \alpha_j$ and thus it is 
enough to study them in the interval $\bigl(0,(n+1)/2\bigr]$. Another interesting feature proved by Varchenko \cite{varchenko1982complex} is that  
the Hodge spectral exponents of $f$ are stable under deformations with constant Milnor number $\mu$.

\vskip 2mm

Recall that a deformation of a hypersurface $f(x_0,\dots,x_n)\in\C\{x_0,\dots,x_n\}$ is a family of hypersurfaces $f_{t_1,\dots,t_k}(x_0,\dots,x_n)$ 
for some set of parameters $(t_1,\dots,t_k)\in S\subseteq\C^k$, satisfying $f(x_0,\dots,x_n)=f_{0,\dots,0}(x_0,\dots,x_n)$. Then, in Varchenko's result we are asking that
the Milnor number of $f_{t_1,\dots,t_k}(x_0,\dots,x_n)$ is the same for all $(t_1,\dots,t_k)\in S$.

\vskip 2mm


K.~Saito \cite{Ksaito1983zeroes} considered the normalized spectrum which he denoted as 
 the \textit{characteristic function} 
   \begin{equation*}
        \Chi_f(T) = \frac{1}{\mu} \sum_{i=1}^\mu T^{\alpha_i}.
    \end{equation*}

We may also display the Hodge spectral exponents as a discrete (probability) distribution on $\R$. Namely, The \textit{distribution of the Hodge spectral exponents} is
    \begin{equation*}
        D_f(s) = \frac{1}{\mu} \sum_{i=1}^\mu \delta(s-\alpha_i) ,
    \end{equation*}
    where $\delta(s)$ is Dirac's delta distribution.
Indeed considering either $D_f(s)$ or $\Chi_f(T)$ is equivalent because the characteristic function is the Fourier transform of the distribution of Hodge spectral exponents, i.e.,
    \begin{equation*}
        \Chi_f(T)
        = \Fourier{D_f(s)}(\tau).
    \end{equation*}
Considering the change of variables $  T=e^{2\pi\iu \tau}$
we will treat the dependence of $\Chi_f$ on $T$ and on $\tau$ interchangeably throughout this paper.

\begin{remark}
 Because of the symmetry of the Hodge spectrum, we will be interested in the truncations
   \begin{equation*}
        \Chi^{<1}_f(T) = \frac{1}{\mu} \sum_{\alpha_i<1} T^{\alpha_i} \hskip 2mm, \hskip 1cm D^{<1}_f(s) = \frac{1}{\mu} \sum_{\alpha_i<1} \delta(s-\alpha_i).
    \end{equation*}
 
\end{remark}

\vskip 2mm

\begin{definition}\label{def:N} The {\it continuous distribution} function is  $  N_{n+1} \colon \R\longrightarrow \R$ defined as:
    \begin{align*}
        N_{n+1}(s)
        &= \int_{x_0+\dots+x_n=s} \I{[0,1)}(x_0) \cdots \I{[0,1)}(x_n) \dd{x_0}\dots\dd{x_n}
        \\&= \left(\I{[0,1)} \ast \overset{n+1}{\dots} \ast \I{[0,1)}\right)(s),
    \end{align*}
where   $\I{[0,1)}(s)$ is the indicator function  and $\ast$ denotes the convolution product. One may check that the Fourier transform of $N_{n+1}(s)$ is
    \begin{equation*}
        \Fourier{N_{n+1}(s)}(\tau)=\left(\frac{T-1}{\log{T}}\right)^{n+1}.
    \end{equation*}
\end{definition}

\begin{figure}[htbp!]
   \centering
   \includegraphics[width=0.6\linewidth]{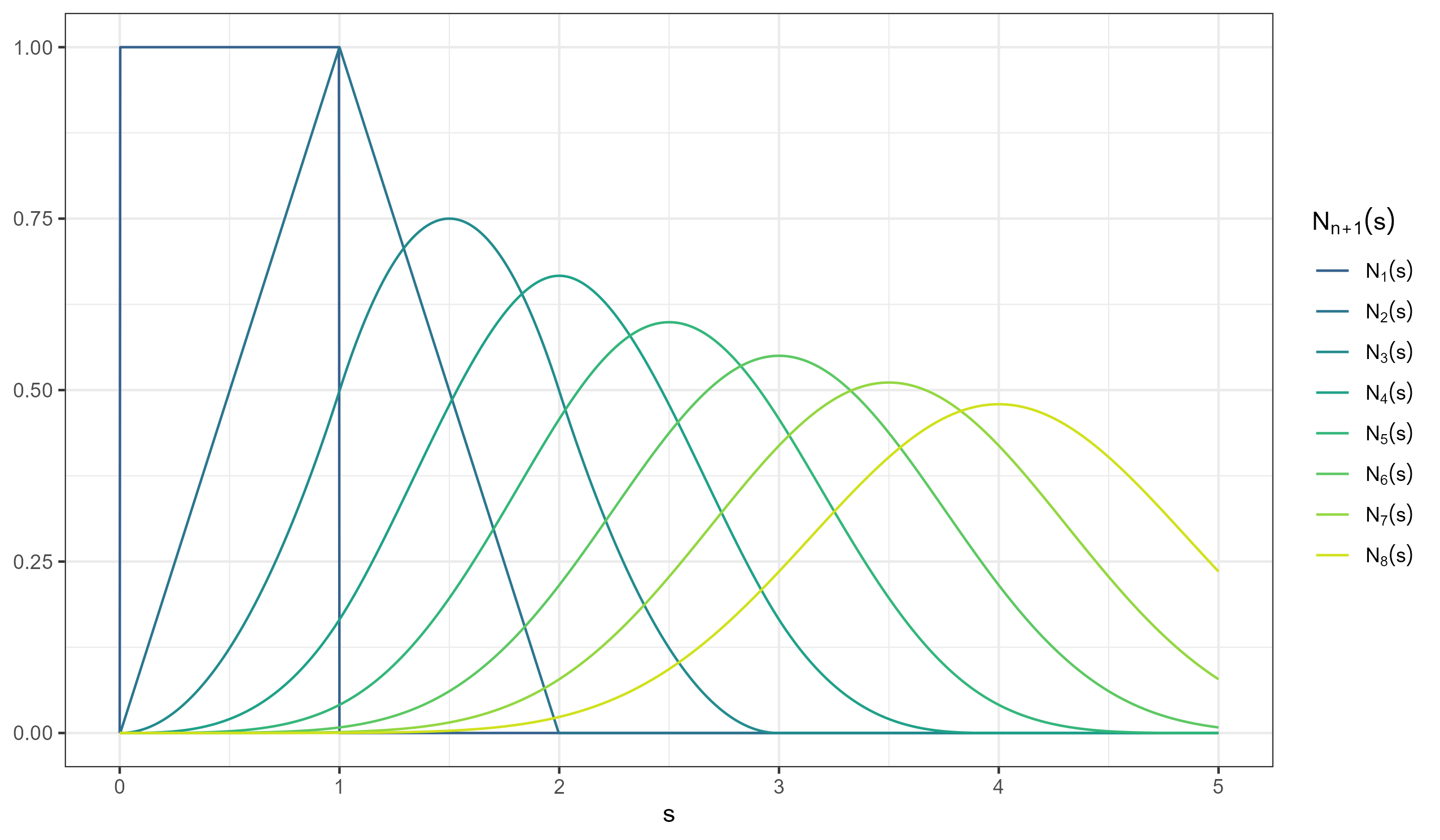}
   \caption{The function $N_{n+1}(s)$ for different values of $n$. 
}
\end{figure}

\vskip 2mm

\begin{definition}\label{def:phi}
We define
      $  \phi_f \colon \left[0,\frac{n+1}{2}\right) \longrightarrow \R$
as the {\it cumulative difference function} between $N_{n+1}(s)$ and $D_f(s)$, that is
    \begin{align*}
        \phi_f(r)
        &= \int_0^r N_{n+1}(s) - D_f(s) \dd{s}
        = \int_0^r N_{n+1}(s) - \frac{1}{\mu} \sum_{i=1}^{\mu} \delta(s-\alpha_i) \dd{s}
        = \int_0^r N_{n+1}(s) \dd{s} - \frac{1}{\mu} \#\{\alpha_i\leqslant r\}.
    \end{align*}
\end{definition}

\begin{definition}\label{def:dominating}
  We say that  $r\in\left[0,\frac{n+1}{2}\right)$ is a \textit{dominating value} if $\phi_f(r) > 0$, or equivalently if
    \begin{equation*}
        \frac{1}{\mu} \#\{\alpha_i\leqslant r\} < \int_0^r N_{n+1}(s) \dd{s}.
    \end{equation*}
\end{definition}

\begin{figure}[htbp!]
    \centering
    \subfloat{\includegraphics[width=7.5cm]{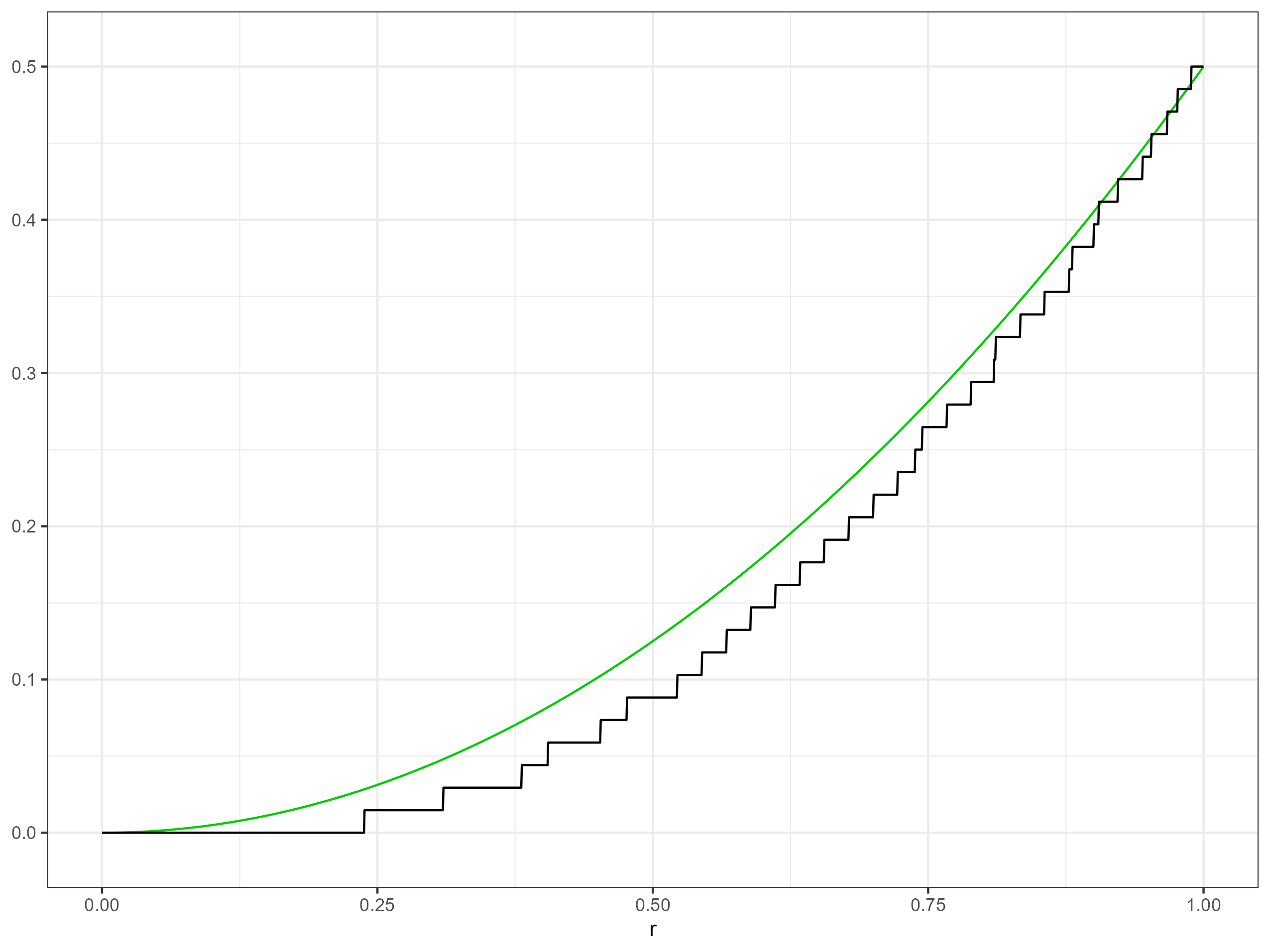} }%
    \qquad
    \subfloat{\includegraphics[width=7.5cm]{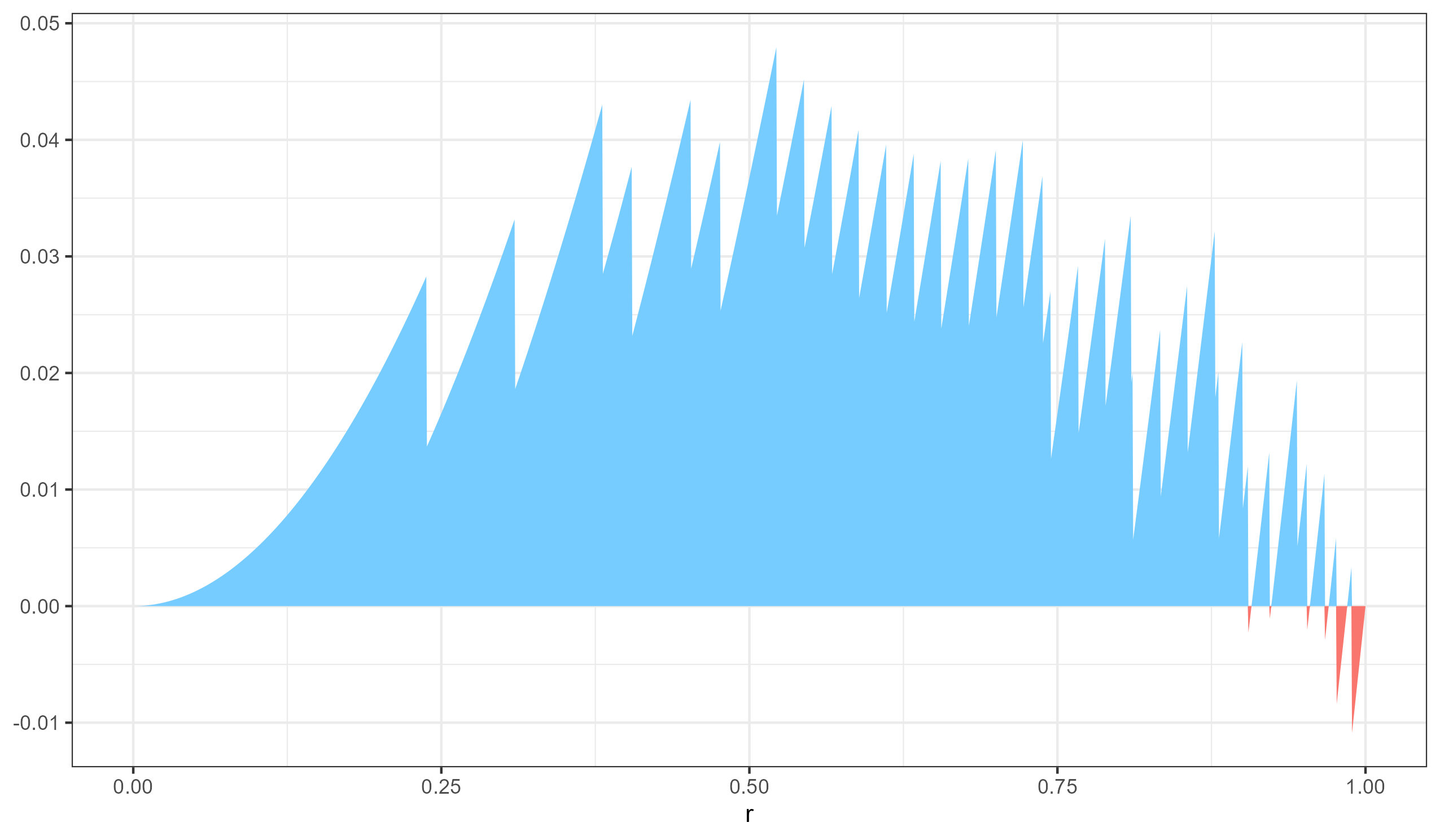} }%
    \caption{ On the left, example of functions $\frac{1}{\mu} \#\{\alpha_i\leqslant r\}$ (black stairs) and $\int_0^r N_{n+1}(s) \dd{s}$ (green line) in the interval $r\in[0,1]$.
    On the right, cumulative difference function $\phi_f(r)$, coloured according to its sign. Both plots corresponds to an irreducible plane curve $f\in\C\{x,y\}$ with Puiseux pairs $(3,4), (2,3)$ in the notation from Definition \ref{def:pairs}.}%
\end{figure}

\vskip 2mm

 K.~Saito  \cite{Ksaito1983zeroes} introduced these notions of cumulative difference function and dominating values. Moreover he formulated the following questions:

\vskip 3mm

\begin{QUESTION}\label{Q1}
For which limits of sequences of hypersurfaces $\left(f^{(i)}\right)_{i\geqslant 0}$ does the distribution of Hodge spectral exponents $D_{f^{(i)}}(s)$ converge to $N_{n+1}(s)$?
   Equivalently, for which limits of $\left(f^{(i)}\right)_{i\geqslant 0}$ does the characteristic function $\Chi_{f^{(i)}}(T)$ converge to $\Fourier{N_{n+1}}(\tau)=\left(\frac{T-1}{\log{T}}\right)^{n+1}$?
\end{QUESTION}

\vskip 3mm
\begin{QUESTION}\label{Q2}
    Given $f$, what is the set of all dominating values?
\end{QUESTION}

\vskip 3mm

The limit of $\left(f^{(i)}\right)_{i\geqslant 0}$ in Question \ref{Q1} has to be specified, since it is not clear a priori which kind of limit one should consider. The few results we may find in the literature all consider different types of limits. K.~Saito  already calculated the following two limits, both of which converge to $N_{n+1}(s)$:

\begin{proposition}[{\cite[(3.7)]{Ksaito1983zeroes}}]
    Let $f\in\C[x_0,\dots,x_n]$ be a quasi-homogeneous polynomial of degree 1 with respect to the weights $r_0,\dots,r_n$, i.e., satisfying $f(\lambda^{r_0}x_0,\dots,\lambda^{r_n}x_n)=\lambda f(x_0,\dots,x_n)$. Then, taking a sequence of such functions with the limit $r_i\rightarrow 0$ for all $i=0,\dots,n$, one has
    \begin{equation*}
        \lim_{r_0,\dots,r_n\rightarrow 0} \Chi_f(T)
        = \left(\frac{T-1}{\log{T}}\right)^{n+1}
        = \Fourier{N_{n+1}(s)}(\tau).
    \end{equation*}
\end{proposition}

\begin{proposition}[{\cite[(3.9)]{Ksaito1983zeroes}}]\label{prop:limitSaito}
    Let $f\in\C\{x,y\}$ be an irreducible plane curve with Puiseux pairs $(n_1,l_1),\dots,(n_g,l_g)$. Then, taking a sequence of such functions with the limit $n_g\rightarrow +\infty$ (keeping all other $n_j$ and $l_j$ fixed), one has
    \begin{equation*}
        \lim_{n_g\rightarrow +\infty} \Chi_f(T)
        = \left(\frac{T-1}{\log{T}}\right)^2
        = \Fourier{N_2(s)}(\tau).
    \end{equation*}
\end{proposition}
\noindent
The Puiseux pairs are defined in Section \ref{section:preliminaries:planeCurves}. 

\vskip 2mm 

More recently, Almir\'on and Schulze  gave another example for which the distribution of Hodge spectral exponents also converges to the continuous distribution $N_{n+1}(s)$:

\begin{proposition}[\cite{almiron2022limit}]
    Consider a fixed Newton diagram $\Gamma$. Let $f_{\omega}\in\C\{x_0,\dots,x_n\}$ be a Newton non-degenerate function with Newton diagram $\omega\Gamma$ (the rescaling of $\Gamma$ by a factor $\omega\in\Q_{>0}$). Then, taking a sequence of such functions with limit $\omega\rightarrow +\infty$, one has
    \begin{equation*}
        \lim_{\omega\rightarrow +\infty} \Chi_{f_{\omega}}(T)
        = \left(\frac{T-1}{\log{T}}\right)^{n+1}
        = \Fourier{N_{n+1}(s)}(\tau).
    \end{equation*}
\end{proposition}
\noindent

\vskip 2mm

We point out that in all of the above examples, the log-canonical threshold of the sequence of hypersurfaces  has limit 0. 
Thus one is tempted to consider this condition in the definition of limit of Question \ref{Q1}. 
In Section \ref{LCT} we will prove that this is in fact a necessary condition but it is not sufficient.

\vskip 2mm

Regarding Question \ref{Q2} on the set of dominating values, Tomari proved the following result which, in terms of dominating values, states the following:
\begin{theorem}[\cite{tomari1993inequality}]\label{theor:1/2}
    Let $f\in\C\{x,y\}$ be a plane curve and  $0< \alpha_1 \leqslant \dots \leqslant \alpha_\mu < 2$ its corresponding Hodge spectral numbers. Then $\frac{1}{2}$ is a dominating value, i.e.,    \begin{equation*}
        \#\left\{\alpha_i\leqslant \frac{1}{2}\right\} < \frac{\mu}{8}.
    \end{equation*}
\end{theorem}
\noindent
In Section \ref{section:dominating} we will give an alternative proof of this theorem when restricted to the case where $f$ is irreducible. K.~Saito asked whether $\frac{1}{2}$ is a dominating value for any 
$f\in\C\{x_0,\dots,x_n\}$, that is
    \begin{equation*}
        \#\left\{\alpha_i\leqslant \frac{1}{2}\right\} < \frac{\mu}{(n+1)!\, 2^{n+1}}.
    \end{equation*}
    
    \vskip 2mm 

A conjecture posed by Durfee states:
\begin{conjecture}[\cite{durfee1978signature}]
    Let $f\in\C\{x,y,z\}$ be surface with a singularity at the origin. Then
    \begin{equation*}
        p_g < \frac{\mu}{6}.
    \end{equation*}
\end{conjecture}
Here, $p_g$ denotes the {\it geometric genus} defined as
 \begin{equation*}
        p_g = \dim_\C \left( R^{n-1} \pi_*\mathcal{O}_{X} \right)_0
         \hskip 4mm \text{for } n\geqslant 2, \hskip 1cm 
        (\; p_g = \dim_\C \left( \pi_*\mathcal{O}_{X}/\mathcal{O}_{\C^{2}} \right)_0
         \hskip 4mm \text{for }n=1 \;),
\end{equation*}
with $\pi\colon X\rightarrow \C^{n+1}$ being a resolution of the singularity. M.~Saito \cite{Msaito40exponents} proved a relation between this invariant and the Hodge spectral exponents. 
Namely, $p_g = \#\{i \,|\, \alpha_i\leqslant 1\}$,  and thus Durfee's conjecture predicts that $1$ a dominating value when $n=2$. 
K.~Saito asked whether one can generalize this statement to:

\vskip 2mm 

\noindent {\bf Question: }
    Is $1$ a dominating value for all $n\geqslant 2$? That is, is it true that
    \begin{equation*}
        p_g = \#\left\{\alpha_i\leqslant 1\right\} < \frac{\mu}{(n+1)!}
    \end{equation*}
    for any $f\in\C\{x_0,\dots,x_n\}$ ?

\vskip 2mm

The aim of this paper is to study Questions \ref{Q1} and \ref{Q2} for the case of irreducible plane curves. 
More precisely,  we will give a closed formula for $\#\{\alpha_i\leqslant r\}$ in Theorem \ref{theor:count} and $\phi_f(r)$ in Theorem \ref{theor:phi} in terms of 
numerical invariants of the  irreducible plane curve. With these tools in hand, we will give  in Theorem \ref{theor:characterization} a characterisation of those sequences of irreducible plane curves whose 
discrete distribution of Hodge spectral exponents converge to the continuous distribution. This would give a complete answer to Question \ref{Q1} in this case. In Section \ref{LCT} we will show that 
the log-canonical threshold converging to zero is a necessary but not sufficient condition for Question \ref{Q1}. Regarding Question \ref{Q2}, we provide in Theorem \ref{theor:intervals} intervals 
of dominating values  for an irreducible plane curve. These intervals contain the value $\frac{1}{2}$ and thus we reprove Tomari's result in this particular case (see Corollary \ref{tomari}).


\section{Irreducible plane curve singularities}\label{section:preliminaries:planeCurves}

In this section we will briefly present the necessary background on irreducible plane curves that we will use in this paper and we  refer to Casas-Alvero's book  \cite{casas2000singularities} for unexplained terminology. 



\vskip 2mm

Let $f\colon (\C^2,0)\rightarrow (\C,0)$ be a germ of a holomorphic function, or equivalently a convergent power series $f\in\C\{x,y\}$. The equation $f=0$ defines locally a (complex) plane curve around the origin. We will only consider irreducible plane curves $f$, i.e., irreducible elements of the unique factorization domain $\C\{x,y\}$. 

\begin{theorem}[Puiseux]
    Let $f\in\C\{x,y\}$ define an irreducible plane curve that is not tangent to the $y$-axis (i.e., $\pdv{f}{x}(0,0)\neq 0$). Then there is a Puiseux series $s(x) = \sum_{i\geqslant 0} a_i x^{i/m}$ such that $f(x,s(x))=0$. Moreover, all such series are conjugates $\sigma_\varepsilon(s) = \sum_{i\geqslant 0} a_i \varepsilon^i x^{i/m}$ with $\varepsilon^m=1$. The curve can be parameterized by $t\mapsto \left(t^m, \sum_{i\geqslant 0} a_i t^i\right)$.
\end{theorem}

A Puiseux series of $f$ has the form
\begin{equation*}
    s(x)
    = \sum_{\substack{j\in (e_0)\\0\leqslant j<\beta_1}} a_j x^{j/m}
    + \sum_{\substack{j\in (e_1)\\\beta_1\leqslant j<\beta_2}} a_j x^{j/m}
    + \dots
    + \sum_{\substack{j\in (e_{g-1})\\\beta_{g-1}\leqslant j<\beta_g}} a_j x^{j/m}
    + \sum_{\substack{j\in (e_g)\\\beta_g\leqslant j}} a_j x^{j/m}
\end{equation*}
with
\begin{align*}
    &e_0 = m
    \\&\beta_i = \min\{j \,|\, a_j\neq 0,\ j\notin (e_{i-1})\} \quad (i=1,\dots,g)
    \\&e_i = \gcd(e_{i-1},\beta_i) \quad (i=1,\dots,g)
\end{align*}
where $m$ is chosen as small as possible, such that $e_g = 1$. Since $e_i | e_{i-1}$, we can define $n_i = e_{i-1} / e_i$ and thus
\begin{equations}\label{def:e}
    &e_i = n_g\cdots n_{i+1} \quad (i=0,\dots,g-1) \\
    &e_g = 1\\
    &n_i \geqslant 2 \quad (i=0,\dots,g-1).
\end{equations}
For the interested reader, these numerical invariants have a geometric meaning: $e_0$ is the multiplicity of $f$ at the origin and $e_i\,(i=1,\dots,g)$ is the multiplicity of $f$ at the $i$-th rupture divisor of its minimal embedded resolution, or equivalently the last infinitely near point of the $i$-th cluster of consecutive satellite points. These concepts are explained in \cite{casas2000singularities}.

\vskip 2mm

We present three equivalent sets of complete topological invariants of the singularity of $f$: 
\begin{itemize}
\item the \textit{characteristic exponents} $\left(\frac{\beta_1}{m},\dots,\frac{\beta_g}{m}\right)$,

\vskip 2mm

\item  the \textit{Puiseux pairs} $(n_1,l_1),\dots,(n_g,l_g)$,

\vskip 2.1mm

\item the \textit{semigroup}    $\ang{\bar\beta_0,\dots,\bar\beta_g}$.

\end{itemize}
They determine, and are determined by, the homeomorphism class of $f^{-1}(0) \cap U$ for a small enough neighbourhood $U$ of the origin. 

\vskip 2mm
From the information given by the Puiseux series we get the first set of invariants.

\begin{definition}
    The \textit{characteristic exponents} of an irreducible plane curve $f$ are the rational numbers $\left(\frac{\beta_1}{m},\dots,\frac{\beta_g}{m}\right)$.
\end{definition}


M.~Saito \cite{Msaito2000exponents} used the following notation for the characteristic exponents:
\begin{equation*}
    \frac{\beta_i}{m} = \frac{k_1}{n_1} + \dots + \frac{k_i}{n_1\cdots n_i}\quad (i=1,\dots,g)
\end{equation*}
with $k_1>n_1$. Additionally, they satisfy $n_j\geqslant 2$, $k_j\geqslant 1$, and $\gcd(k_j,n_j)=1$.

For our purposes it will be more convenient to avoid having to work with the condition $k_1 > n_1$. Therefore we
introduce the following change of variables:
\begin{align*}
    &k_1 = n_1 + l_1 \quad \text{with} \quad l_1\geqslant 1,\\
    &k_j = l_j \quad (j=2,\dots,g).
\end{align*}
 With this slightly different notation, the characteristic exponents are
\begin{equation*}
    \frac{\beta_i}{m} = 1 + \frac{l_1}{n_1} + \dots + \frac{l_i}{n_1\cdots n_i}\quad (i=1,\dots,g)
\end{equation*}
with $n_j\geqslant 2$, $l_j\geqslant 1$, $\gcd(l_j,n_j)=1$.

\begin{definition}\label{def:pairs}
The \textit{Puiseux pairs} of an irreducible plane curve $f$ are $(n_1,l_1),\dots,(n_g,l_g)$.
\end{definition}

\begin{remark}
    The name Puiseux pairs appear in various slightly different ways in the literature. 
    We based our definition on the one given by M.~Saito \cite{Msaito2000exponents}, who used this name for the pairs $(k_1,n_1),\dots,(k_g,n_g)$. 
    Casas-Alvero \cite{casas2000singularities} used the similar term characteristic pairs to refer to $(\beta_1,m),\dots,(\beta_g,m)$. 
    
\end{remark}


\begin{definition}
    The semigroup of an irreducible plane curve $f$ is
    \begin{equation*}
        \Gamma_f = \{ o_f(g) \in\Z_{\geqslant 0} \,|\, g\in\C\{x,y\}\setminus (f) \}
    \end{equation*}
    where the vanishing order of $g$ along $f$ is $  o_f(g) = o_t\Bigl(g\bigl(t^m,s(t^m)\bigr)\Bigr)$
    with $o_t$ the order of a power series in $t$, and $s$ a Puiseux series of $f$.
\end{definition}

The semigroup of $f$ is a finitely generated (additive) semigroup, where the number of the minimal generators is one more than the number of characteristic exponents. 
They are denoted by $  \Gamma_f = \ang{\bar\beta_0,\dots,\bar\beta_g}$. The relation with the characteristic exponents and thus the Puiseux pairs is given by the following:

\begin{proposition}
    The minimal generators of the semigroup $\Gamma_f = \ang{\bar\beta_0,\dots,\bar\beta_g}$ satisfy:
    \begin{align*}
        \bar\beta_0 = m, \hskip 5mm
       \bar\beta_1 = \beta_1, &  \hskip 5mm
        \bar\beta_i = n_{i-1} \bar\beta_{i-1} - \beta_{i-1} + \beta_i \quad (i=2,\dots,g),  \hskip 5mm
        e_i = \gcd(e_{i-1},\bar\beta_i). &
    \end{align*}
\end{proposition}

\subsection{Some juggling with numerical invariants}

M.~Saito \cite{Msaito2000exponents} considered the following numerical invariants
in order to obtain a formula for the characteristic function of the Hodge spectral exponents of an irreducible plane curve.
\begin{align*}
    & w_1 = n_1 + l_1, \\
    & w_j = n_j n_{j-1} w_{j-1} + l_j \quad (j=2,\dots,g).
\end{align*}
We can actually extend his definition by letting $ n_0 = 1 ,\; w_0 = 1,$
from which $ w_1
    = n_1 + l_1
    = n_1 n_0 w_0 + l_1$
also satisfies the recurrence relation. Thus, we can write
\begin{definition}\label{def:w}
    We define the following numerical invariants:
    \begin{align*}
        & w_0 = 1, \\
        & w_j = n_j n_{j-1} w_{j-1} + l_j \quad (j=1,\dots,g).
    \end{align*}
\end{definition}
This means that we can use the recursive definition also in the case $j=1$, which will simplify things later. Checking the recurrence relations, one finds that
\begin{proposition}
    The numerical invariants $w_j$ can be interpreted in terms of the semigroup of $f$ as follows:
    \begin{equation*}
        w_j = \frac{\bar\beta_j}{e_j} \quad (j=0,\dots,g).
    \end{equation*}
\end{proposition}

\begin{definition}\label{def:mu}
    We define the following numerical invariants:
    \begin{align*}
        & \mu_0 = 0, \\
        & \mu_j = (n_j-1)(w_j-1)+n_j\mu_{j-1} \quad (j=1,\dots,g). \\
    \end{align*}
\end{definition}

\begin{proposition}[\cite{Ksaito1983zeroes}]
    The Milnor number of $f$ is $\mu = \mu_g$. More generally,  the Milnor number of a curve with Puiseux pairs $(n_1,l_1),\dots,(n_j,l_j)$ is $\mu_j$, for any $j\in\{1,\dots,g\}$.
\end{proposition}

The following technical lemmas where we express $w_j$ and $n_g \mu_{g-1}$ in terms of $l_j, n_j$ (and $e_j$) will be useful in Section \ref{Sec3}.
\begin{lemma}\label{lemma:wj}
    The numerical invariant $w_j$ for $j=1,\dots,g$ can be written as
    \begin{equation*}
        w_j
        = \sum_{k=1}^j l_k \frac{e_k^2 n_k}{e_j^2 n_j} + \frac{e_0^2}{e_j^2 n_j},
    \end{equation*}
    and in particular
    \begin{align*}
        w_g
        = \sum_{k=1}^g l_k \frac{e_k^2 n_k}{n_g} + \frac{e_0^2}{n_g}.
    \end{align*}
\end{lemma}

\begin{proof}
    We prove it by induction:
    \begin{itemize}
        \item For $j=1$, \begin{align*}
            w_1
            = l_1 \frac{e_1^2 n_1}{e_1^2 n_1} + \frac{e_0^2}{e_1^2 n_1}
            = l_1 + n_1.
        \end{align*}
        \item The induction step: \begin{align*}
            w_j
            &= n_j n_{j-1} w_{j-1} + l_j
            = n_j n_{j-1} \sum_{k=1}^{j-1} l_k \frac{e_k^2 n_k}{e_{j-1}^2 n_{j-1}} + n_j n_{j-1} \frac{e_0^2}{e_{j-1}^2 n_{j-1}} + l_j
            \\&= \sum_{k=1}^{j-1} l_k \frac{e_k^2 n_k}{e_j^2 n_j} + \frac{e_0^2}{e_j^2 n_j} + l_j \frac{e_j^2 n_j}{e_j^2 n_j}
           = \sum_{k=1}^j l_k \frac{e_k^2 n_k}{e_j^2 n_j} + \frac{e_0^2}{e_j^2 n_j}.
        \end{align*}
    \end{itemize}
\end{proof}

\begin{lemma}\label{lemma:ngmug}
    The product of the numerical invariants $n_g \mu_{g-1}$ can be written as
    \begin{align*}
        n_g \mu_{g-1}
        &= \sum_{k=1}^{g-1}  l_k \left(\frac{e_k^2 n_k}{n_g} - \frac{e_k^2 n_k}{e_{k-1}}\right)
            + \frac{e_0^2}{n_g}
            + n_g
            - 2 e_0.
    \end{align*}
\end{lemma}

\begin{proof}
    Using the definition of $\mu_j$ and the previous lemma:
    \begin{align*}
        n_g \mu_{g-1}
        &= \sum_{j=1}^{g-1} (w_j -1) (n_j -1) e_j
        = \sum_{j=1}^{g-1} \left(\sum_{k=1}^j l_k \frac{e_k^2 n_k}{e_j^2 n_j} + \frac{e_0^2}{e_j^2 n_j} -1\right) (n_j -1) e_j
        \\&= \sum_{j=1}^{g-1} \sum_{k=1}^j l_k \frac{e_k^2 n_k}{e_j^2 n_j} (n_j -1) e_j
            + \sum_{j=1}^{g-1} \frac{e_0^2}{e_j^2 n_j} (n_j -1) e_j
            - \sum_{j=1}^{g-1} (n_j -1) e_j.
    \end{align*}
    We can simplify this expression as follows:
    \begin{enumerate}[label={\arabic*)}]
        \item The first term:
            \begin{align*}
                \sum_{j=1}^{g-1} \sum_{k=1}^j l_k \frac{e_k^2 n_k}{e_j^2 n_j} (n_j -1) e_j
                &= \sum_{k=1}^{g-1} \sum_{j=k}^{g-1} l_k \frac{e_k^2 n_k}{e_j^2 n_j} (n_j -1) e_j
                = \sum_{k=1}^{g-1}  l_k e_k^2 n_k \sum_{j=k}^{g-1} \frac{1}{e_j n_j} (n_j -1)
                \\&= \sum_{k=1}^{g-1}  l_k e_k^2 n_k \sum_{j=k}^{g-1} \left(\frac{1}{e_j} - \frac{1}{e_{j-1}}\right)
                = \sum_{k=1}^{g-1}  l_k e_k^2 n_k \left(\frac{1}{e_{g-1}} - \frac{1}{e_{k-1}}\right)
                \\&= \sum_{k=1}^{g-1}  l_k \left(\frac{e_k^2 n_k}{n_g} - \frac{e_k^2 n_k}{e_{k-1}}\right).
            \end{align*}
        \item The second term:
            \begin{align*}
                \sum_{j=1}^{g-1} \frac{e_0^2}{e_j^2 n_j} (n_j -1) e_j
                = \sum_{j=1}^{g-1} e_0^2 \left(\frac{1}{e_j} - \frac{1}{e_{j-1}}\right)
                = e_0^2 \left(\frac{1}{e_{g-1}} - \frac{1}{e_0}\right)
                =  \frac{e_0^2}{n_g} - e_0.
            \end{align*}
        \item The third term:
            \begin{align*}
                - \sum_{j=1}^{g-1} (n_j -1) e_j
                = \sum_{j=1}^{g-1} e_j - e_{j-1}
                = e_{g-1} - e_0
                = n_g - e_0.
            \end{align*}
    \end{enumerate}
    
    Putting it together:
    \begin{align*}
        n_g \mu_{g-1}
        &= \sum_{k=1}^{g-1}  l_k \left(\frac{e_k^2 n_k}{n_g} - \frac{e_k^2 n_k}{e_{k-1}}\right)
            + \frac{e_0^2}{n_g}
            + n_g
            - 2 e_0.
    \end{align*}
\end{proof}

\begin{lemma}\label{lemma:mu}
    The Milnor number of an irreducible plane curve with Puiseux pairs $(n_1,l_1),\dots,(n_g,l_g)$ is
    \begin{equation*}
        \mu
        = \sum_{j=1}^g l_j e_j (e_{j-1}-1) + (e_0-1)^2.
    \end{equation*}
\end{lemma}
\noindent


\begin{proof}
    From Lemma \ref{lemma:wj} we have:
    \begin{equation*}
        w_j
        = \sum_{k=1}^j l_k \frac{e_k^2 n_k}{e_j^2 n_j} + \frac{e_0^2}{e_j^2 n_j}.
    \end{equation*}
    Then, the Milnor number can be written as
    \begin{align*}
        \mu
        &= \mu_g
        = \sum_{j=1}^g (w_j -1) (n_j -1) e_j
        \\&= \sum_{j=1}^g \sum_{k=1}^j l_k \frac{e_k^2 n_k}{e_j^2 n_j} (n_j -1) e_j
            + \sum_{j=1}^g \frac{e_0^2}{e_j^2 n_j} (n_j -1) e_j
            - \sum_{j=1}^g (n_j -1) e_j.
    \end{align*}
    We can simplify this expression as follows.
    \begin{enumerate}[label={\arabic*)}]
        \item The first term:
            \begin{align*}
                \sum_{j=1}^g \sum_{k=1}^j l_k \frac{e_k^2 n_k}{e_j^2 n_j} (n_j -1) e_j
                &= \sum_{k=1}^g \sum_{j=k}^g l_k \frac{e_k^2 n_k}{e_j^2 n_j} (n_j -1) e_j
                = \sum_{k=1}^g  l_k e_k^2 n_k \sum_{j=k}^g \frac{1}{e_j n_j} (n_j -1)
                \\&= \sum_{k=1}^g  l_k e_k^2 n_k \sum_{j=k}^g \left(\frac{1}{e_j} - \frac{1}{e_{j-1}}\right)
                = \sum_{k=1}^g  l_k e_k^2 n_k \left(\frac{1}{e_g} - \frac{1}{e_{k-1}}\right)
                \\&= \sum_{k=1}^g  l_k e_k e_{k-1} \left(1 - \frac{1}{e_{k-1}}\right)
                = \sum_{k=1}^g  l_k e_k (e_{k-1} - 1).
            \end{align*}
        \item The second term:
            \begin{align*}
                \sum_{j=1}^g \frac{e_0^2}{e_j^2 n_j} (n_j -1) e_j
                = e_0^2 \sum_{j=1}^g \left(\frac{1}{e_j} - \frac{1}{e_{j-1}}\right)
                = e_0^2 \left(\frac{1}{e_g} - \frac{1}{e_0}\right)
                =  e_0^2 - e_0.
            \end{align*}
        \item The third term:
            \begin{align*}
                - \sum_{j=1}^g (n_j -1) e_j
                = \sum_{j=1}^g e_j - e_{j-1}
                = e_g - e_0
                = 1 - e_0.
            \end{align*}
    \end{enumerate}
    Putting it together:
    \begin{equation*}
        \mu
        = \sum_{k=1}^g  l_k e_k (e_{k-1} - 1) + e_0^2 - e_0 + 1 - e_0 
        = \sum_{j=1}^g  l_j e_j (e_{j-1} - 1) + (e_0 - 1)^2.
    \end{equation*}
\end{proof}

\subsection{Hodge spectral exponents of irreducible plane curves}

Th\`an and Steenbrink \cite{thanh1989spectre} already described the Hodge spectrum of any plane curve in terms its topological invariants, but in this work we will use a closed formula
given by M.~Saito: 


\begin{theorem}[{\cite[Th. 1.5]{Msaito2000exponents}}]\label{cor:exponents}
    The Hodge spectral exponents in the interval $(0,1)$ are:
   
   \begin{equation*}
  \left\{ \frac{1}{e_j} \left(\frac{b}{n_j}+\frac{a}{w_j}\right) + \frac{c}{e_j} \ \middle|\  \begin{array}{l}
        0 < a < w_j \\
        0 < b < n_j \\
        0 \leqslant c < n_{j+1}\cdots n_g \\
        1 \leqslant j \leqslant g
    \end{array}, \frac{b}{n_j}+\frac{a}{w_j} < 1\right\}.
\end{equation*}
 
\end{theorem}

Notice that this formula gives us a set of $\mu/2$ Hodge spectral exponents and thus, by symmetry, it characterizes all the Hodge spectral exponents of $f$.

\begin{remark}
If we use
\begin{align*}
    & 1\leqslant a\leqslant w_j-1 \text{ and } \frac{b}{n_j} + \frac{a}{w_j} < 1
    \\&\iff 1\leqslant a\leqslant w_j-1 \text{ and } a < w_j\left(1 - \frac{b}{n_j}\right)
    \\&\iff 1\leqslant a\leqslant \floor*{w_j\left(1 - \frac{b}{n_j}\right)},
\end{align*}
the set of Hodge spectral exponents in $(0,1)$ can be written as
\begin{align*}\label{eq:setExponents}
    \bigcup_{j=1}^g \bigcup_{c=0}^{e_j-1} \bigcup_{b=1}^{n_j-1} \left\{ \frac{c}{e_j} + \frac{b}{e_j n_j} + \frac{a}{e_j w_j} \ \middle|\  1\leqslant a\leqslant\floor*{w_j\left(1-\frac{b}{n_j}\right)} \right\}.
\end{align*}
\end{remark}
In terms of the distribution,
\begin{align*}
    \mu\, D_f^{<1}(s)
    &= \sum_{j=1}^g \sum_{c=0}^{e_j-1} \sum_{b=1}^{n_j-1} \sum_{a=1}^{\floor*{w_j\left(1-\frac{b}{n_j}\right)}} \delta\left(s - \left( \frac{c}{e_j} + \frac{b}{e_j n_j} + \frac{a}{e_j w_j} \right)\right).
\end{align*}

\begin{corollary}\label{prop:lct}
    The smallest Hodge spectral exponent is the log-canonical threshold
    \begin{equation*}
        \alpha_1= \lct(f)
        = \frac{1}{e_1} \left(\frac{1}{n_1}+\frac{1}{n_1+l_1}\right) = \frac{1}{n_g\cdots n_2} \left(\frac{1}{n_1}+\frac{1}{n_1+l_1}\right) = \frac{1}{e_0} \left(1+\frac{1}{1+\frac{l_1}{n_1}}\right).
    \end{equation*}
    The largest  Hodge spectral exponent in the interval $(0,1)$ is 
\begin{align*}
  \alpha_{\mu/2}=   1 - \frac{1}{e_{g-1} w_g}
    = 1 - \frac{1}{n_g w_g}.
\end{align*}
\end{corollary}

\section{Cumulative difference function \texorpdfstring{$\phi_f$}{ϕ\_f}} \label{Sec3}



From Definition \ref{def:phi} we have that the cumulative difference function for the case of plane curves is $\phi_f \colon [0,1) \longrightarrow \R$ defined as

\begin{align*}
    \phi_f(r)
    &= \int_0^r N_2(s) - D_f(s) \dd{s}
    = \int_0^r N_2(s) - \frac{1}{\mu} \sum_{i=1}^{\mu} \delta(s-\alpha_i) \dd{s}
     = \frac{1}{2} r^2 - \frac{1}{\mu} \#\{\alpha_i\leqslant r\}
\end{align*}
since we have $N_2(s)=s$ in the interval $[0,1)$.

\vskip 2mm

In this section we will prove the following explicit formulas:

\begin{theorem}\label{theor:count}
    Let $f$ be an irreducible plane curve with Puiseux pairs $(n_1,l_1),\dots,(n_g,l_g)$. Then, for any $r\in[0,1)$, the number of Hodge spectral exponents less or equal to $r$ is given by the following expression:
    \begin{align*}
        \#\{\alpha_i\leqslant r\}
        &= \frac{\mu_g - n_g w_g}{2} r
            + \frac{n_g w_g}{2} r^2
            + \sum_{j=1}^g \frac{n_j -1}{2} \fract{e_j r}
            + \frac{1}{2} \fract{e_0 r} (1 - \fract{e_0 r})
            \\&\quad + \sum_{j=1}^g \frac{l_j}{2n_j} \fract{e_{j-1} r}(1-\fract{e_{j-1} r})
            - \sum_{j=1}^g \sum_{b=1}^{n_j-1} \fract*{w_j\left(\fract{e_j r} - \frac{b}{n_j}\right)}
                \I{\left[\left.\frac{b}{n_j},1\right)\right.}\!\left(\fract{e_j r}\right).
    \end{align*}
\end{theorem}

\begin{theorem}\label{theor:phi}
    Let $f$ be an irreducible plane curve with Puiseux pairs $(n_1,l_1),\dots,(n_g,l_g)$. Then, for any $r\in[0,1)$, the cumulative difference function between $N_2(s)$ and $D_f(s)$ is given by the following expression:
    \begin{align*}
        \phi_f(r)
        &= \frac{1}{2\mu} \left( \biggl(2e_0 - 1 + \sum_{j=1}^g  l_j e_j\biggr) r(1-r)
            - \sum_{j=1}^g (n_j -1) \fract{e_j r}
            - \fract{e_0 r} (1 - \fract{e_0 r})
            \right.\\&\left.\hspace{15mm} - \sum_{j=1}^g \frac{l_j}{n_j} \fract{e_{j-1} r}(1-\fract{e_{j-1} r})
            + \sum_{j=1}^g \sum_{b=1}^{n_j-1} 2 \fract*{w_j\left(\fract{e_j r} - \frac{b}{n_j}\right)} \I{\left[\frac{b}{n_j},1\right)}\!\left(\fract{e_j r}\right)
            \right).
    \end{align*}
\end{theorem}

\subsubsection{Proof of Theorem \ref{theor:count}}


The number of Hodge spectral exponents 
less or equal to $r$
for $r\in [0,1)$ is given by the cumulative distribution, that is
\begin{equation*}
    \#\{\alpha_i\leqslant r\}
    = \int_0^r \mu\, D_f^{<1}(s) \dd{s}.
\end{equation*}
To calculate this integral, we start with the following observation.
    Let $m\in\Z_{>0}$. Then
    \begin{align*}
        \#\{i\in\Z \,|\, 1\leqslant i \leqslant m,\ i\leqslant x \}
        = \floor{x} \I{[0,m)}(x) + m \I{[m,+\infty)}(x),
    \end{align*}
    and since $\floor{x}=m$ in the interval $x\in[m,m+1)$, we also have
    \begin{align*}
        \#\{i\in\Z \,|\, 1\leqslant i \leqslant m,\ i\leqslant x \}
        = \floor{x} \I{[0,\lambda)}(x) + m \I{[\lambda,+\infty)}(x)
    \end{align*}
    for any $\lambda \in[m,m+1]$.
Then we have:
\begin{align*}
    \#\{\alpha_i\leqslant r\}
    &= \int_0^r \mu\, D_f^{<1}(s) \dd{s}
    \\&= \int_0^r \sum_{j=1}^g \sum_{c=0}^{e_j-1} \sum_{b=1}^{n_j-1} \sum_{a=1}^{\floor*{w_j\left(1-\frac{b}{n_j}\right)}} \delta\left(s - \left( \frac{c}{e_j} + \frac{b}{e_j n_j} + \frac{a}{e_j w_j} \right)\right) \dd{s}
    \\&= \sum_{j=1}^g \sum_{c=0}^{e_j-1} \sum_{b=1}^{n_j-1} \#\left\{a\in\Z \,\middle|\, 1\leqslant a\leqslant \floor*{w_j\left(1-\frac{b}{n_j}\right)},\  \frac{c}{e_j} + \frac{b}{e_j n_j} + \frac{a}{e_j w_j} \leqslant r \right\}
    \\&= \sum_{j=1}^g \sum_{c=0}^{e_j-1} \sum_{b=1}^{n_j-1} \#\left\{a\in\Z \,\middle|\, 1\leqslant a\leqslant \floor*{w_j\left(1-\frac{b}{n_j}\right)},\  a \leqslant w_j\left(e_j r - c - \frac{b}{n_j}\right) \right\}
    \\&= \sum_{j=1}^g \sum_{c=0}^{e_j-1} \sum_{b=1}^{n_j-1} \floor*{w_j\left(e_j r - c - \frac{b}{n_j}\right)} \I{\left[0,w_j\left(1-\frac{b}{n_j}\right)\right)}\!\left(w_j\left(e_j r - c - \frac{b}{n_j}\right)\right)
        \\&\qquad + \floor*{w_j\left(1-\frac{b}{n_j}\right)} \I{\left[w_j\left(1-\frac{b}{n_j}\right),+\infty\right)}\!\left(w_j\left(e_j r - c - \frac{b}{n_j}\right)\right)
    \\&= \sum_{j=1}^g \sum_{c=0}^{e_j-1} \sum_{b=1}^{n_j-1} \floor*{w_j\left(e_j r - c - \frac{b}{n_j}\right)} \I{\left[\frac{b}{n_j},1\right)}\!\left(e_j r - c\right)
    + \floor*{w_j\left(1-\frac{b}{n_j}\right)} \I{[1,+\infty)}\!\left(e_j r - c\right).
\end{align*}

\vskip 2mm 

Now we need the following facts:

\vskip 2mm 
\begin{enumerate}[label={\arabic*)}]
    \item For the first term, let $f(x)$ be a function with support in $[0,1)$.  Since
        \begin{align*}
            &x-c\in [0,1) \iff x\in [c,c+1) \iff c=\floor{x},
            &0\leqslant c\leqslant e-1 \iff x\in [0,e),
        \end{align*} then we have
        \begin{equation*}
            \sum_{c=0}^{e-1} f(x-c)
            = f(x - \floor{x}) \I{[0,e)}(x)
            = f(\fract{x}) \I{[0,e)}(x)
        \end{equation*}
    \item For the second term, we calculate
        \begin{align*}
            \sum_{c=0}^{e-1} \I{[1,+\infty)}(x - c)
            &= \sum_{c=0}^{e-1} \I{[c,+\infty)}(x - 1)
            = \sum_{c=1}^{e} \I{[c,+\infty)}(x)
            = \#\{c\in\Z \,|\, 1\leqslant i \leqslant e,\ c\leqslant x \}
            \\&= \floor{x} \I{[0,e)}(x) + e\,\I{[e,+\infty)}(x).
        \end{align*}
        
    \item Also for the second term, we have
        \begin{align*}
            \sum_{b=1}^{n_j-1} \floor*{w_j\left(1-\frac{b}{n_j}\right)}
            &= \sum_{b=1}^{n_j-1} \sum_{a=1}^{\floor*{w_j\left(1-\frac{b}{n_j}\right)}} 1
            = \#\left\{\frac{b}{n_j}+\frac{a}{w_j}<1 \,\middle|\, 1\leq b \leq n_j-1,\, 1\leq a \leq w_j-1\right\}
            \\&= \frac{(n_j -1)(w_j -1)}{2}.
        \end{align*}
\end{enumerate}

Using these calculations, we continue:
\begin{align*}
    \#\{\alpha_i\leqslant r\}
    &= \sum_{j=1}^g \sum_{b=1}^{n_j-1} \floor*{w_j\left(\fract{e_j r} - \frac{b}{n_j}\right)} \I{\left[\frac{b}{n_j},1\right)}\!\left(\fract{e_j r}\right)
        \\&\qquad + \floor*{w_j\left(1-\frac{b}{n_j}\right)} \left( \floor{e_j r} \I{[0,e_j)}(e_j r) + e_j\,\I{[e_j,+\infty)}(e_j r) \right)
    \\&= \sum_{j=1}^g \sum_{b=1}^{n_j-1} \floor*{w_j\left(\fract{e_j r} - \frac{b}{n_j}\right)} \I{\left[\frac{b}{n_j},1\right)}\!\left(\fract{e_j r}\right)
        \\&\qquad + \sum_{j=1}^g \sum_{b=1}^{n_j-1} \floor*{w_j\left(1-\frac{b}{n_j}\right)} \floor{e_j r}
    \\&= \sum_{j=1}^g \sum_{b=1}^{n_j-1} \left(w_j\left(\fract{e_j r} - \frac{b}{n_j}\right) - \fract*{w_j\left(\fract{e_j r} - \frac{b}{n_j}\right)}\right)\I{\left[\frac{b}{n_j},1\right)}\!\left(\fract{e_j r}\right)
        \\&\qquad + \sum_{j=1}^g \frac{(n_j -1)(w_j -1)}{2} (e_j r - \fract{e_j r}).
\end{align*}

Now, we use the following facts:
\vskip 2mm
\begin{enumerate}[label={\arabic*)}]
    \item For the first term, we calculate
        \begin{align*}
            \sum_{b=1}^{n-1} \left(x-\frac{b}{n}\right) \I{\left[\frac{b}{n},1\right)}(x)
            &= \floor{n x} x - \frac{1}{n} \sum_{b=1}^{n-1} b \I{[b,n)}(n x)
            = \floor{n x} x - \frac{1}{n} \sum_{b=0}^{\floor{n x}} b
            \\&= \floor{n x} x - \frac{1}{n} \frac{(\floor{n x}+1)\floor{n x}}{2}
            = \floor{n x} \left(x - \frac{\floor{n x}+1}{2 n}\right)
            \\&= (n x - \fract{n x}) \left(\frac{x}{2} - \frac{1-\fract{n x}}{2 n}\right)
            = \frac{n}{2} x^2 - \frac{1}{2} x + \frac{1}{2n}\fract{n x}(1-\fract{n x}). 
        \end{align*}
    \item Also for the first term, we use 
        \begin{align*}
            \fract{n_j\fract{e_jr}}
            = n_j\fract{e_jr} - \floor{n_je_jr-n_j\floor{e_jr}}
            = n_je_jr - \floor{n_je_jr}
            = \fract{n_je_jr}
            = \fract{e_{j-1}r}.
        \end{align*}
    \item For the second term,
        \begin{align*}
            \sum_{j=1}^g \frac{(n_j -1)(w_j -1)}{2} e_j
            = \frac{\mu_g}{2}.
        \end{align*}
\end{enumerate}

With this, we finish our calculations:

\begin{align*}
    &\hspace{-2mm}\#\{\alpha_i\leqslant r\} =
    \\&= \sum_{j=1}^g w_j \sum_{b=1}^{n_j-1} \left(\fract{e_j r} - \frac{b}{n_j}\right) \I{\left[\frac{b}{n_j},1\right)}\!\left(\fract{e_j r}\right)
    - \sum_{j=1}^g \sum_{b=1}^{n_j-1} \fract*{w_j\left(\fract{e_j r} - \frac{b}{n_j}\right)} \I{\left[\frac{b}{n_j},1\right)}\!\left(\fract{e_j r}\right)
        \\&\qquad + \sum_{j=1}^g \frac{(n_j -1)(w_j -1)}{2} (e_j r - \fract{e_j r})
    \\&= \sum_{j=1}^g \frac{w_j n_j}{2} \fract{e_j r}^2
        - \sum_{j=1}^g \frac{w_j}{2} \fract{e_j r}
        + \sum_{j=1}^g \frac{w_j}{2n_j}\fract{e_{j-1} r}(1-\fract{e_{j-1} r})
        \\&\qquad - \sum_{j=1}^g \sum_{b=1}^{n_j-1} \fract*{w_j\left(\fract{e_j r} - \frac{b}{n_j}\right)} \I{\left[\frac{b}{n_j},1\right)}\!\left(\fract{e_j r}\right)
        + \sum_{j=1}^g \frac{(n_j -1)(w_j -1)}{2} (e_j r - \fract{e_j r})
    \\&= \sum_{j=1}^g \frac{(n_j -1)(w_j -1)}{2} e_j r
        - \sum_{j=1}^g \frac{n_j w_j - n_j - w_j +1}{2} \fract{e_j r}
        + \sum_{j=1}^g \frac{w_j n_j}{2} \fract{e_j r}^2
        - \sum_{j=1}^g \frac{w_j}{2} \fract{e_j r}
        \\&\qquad + \sum_{j=1}^g \frac{w_j}{2n_j}\fract{e_{j-1} r}(1-\fract{e_{j-1} r})
        - \sum_{j=1}^g \sum_{b=1}^{n_j-1} \fract*{w_j\left(\fract{e_j r} - \frac{b}{n_j}\right)} \I{\left[\frac{b}{n_j},1\right)}\!\left(\fract{e_j r}\right)
    \\&= \frac{\mu_g}{2} r
        + \sum_{j=1}^g \frac{n_j -1}{2} \fract{e_j r}
        - \sum_{j=1}^g \frac{n_j w_j}{2} \fract{e_j r} (1 - \fract{e_j r})
        \\&\qquad + \sum_{j=1}^g \left(\frac{n_{j-1} w_{j-1}}{2} + \frac{l_j}{2n_j}\right)\fract{e_{j-1} r}(1-\fract{e_{j-1} r})
        - \sum_{j=1}^g \sum_{b=1}^{n_j-1} \fract*{w_j\left(\fract{e_j r} - \frac{b}{n_j}\right)} \I{\left[\frac{b}{n_j},1\right)}\!\left(\fract{e_j r}\right)
    \\&= \frac{\mu_g}{2} r
        + \sum_{j=1}^g \frac{n_j -1}{2} \fract{e_j r}
        - \frac{n_g w_g}{2} \fract{e_g r} (1 - \fract{e_g r})
        + \frac{n_0 w_0}{2} \fract{e_0 r} (1 - \fract{e_0 r})
        \\&\qquad + \sum_{j=1}^g \frac{l_j}{2n_j} \fract{e_{j-1} r}(1-\fract{e_{j-1} r})
        - \sum_{j=1}^g \sum_{b=1}^{n_j-1} \fract*{w_j\left(\fract{e_j r} - \frac{b}{n_j}\right)} \I{\left[\frac{b}{n_j},1\right)}\!\left(\fract{e_j r}\right)
    \\&= \frac{\mu_g - n_g w_g}{2} r
        + \frac{n_g w_g}{2} r^2
        + \sum_{j=1}^g \frac{n_j -1}{2} \fract{e_j r}
        + \frac{1}{2} \fract{e_0 r} (1 - \fract{e_0 r})
        \\&\qquad + \sum_{j=1}^g \frac{l_j}{2n_j} \fract{e_{j-1} r}(1-\fract{e_{j-1} r})
        - \sum_{j=1}^g \sum_{b=1}^{n_j-1} \fract*{w_j\left(\fract{e_j r} - \frac{b}{n_j}\right)} \I{\left[\frac{b}{n_j},1\right)}\!\left(\fract{e_j r}\right).
\end{align*}


This finishes the proof of Theorem \ref{theor:count}.

\subsubsection{Proof of Theorem \ref{theor:phi}}


We want to calculate
\begin{align*}
    \phi_f(r)
    &= \frac{1}{2} r^2 - \frac{1}{\mu} \#\{\alpha_i\leqslant r\}
    = \frac{1}{\mu} \left(\frac{\mu}{2} r^2 - \#\{\alpha_i\leqslant r\}\right).
\end{align*}
Theorem \ref{theor:count} states
\begin{align*}
    \#\{\alpha_i\leqslant r\}
    &= \frac{\mu_g - n_g w_g}{2} r
        + \frac{n_g w_g}{2} r^2
        + \sum_{j=1}^g \frac{n_j -1}{2} \fract{e_j r}
        + \frac{1}{2} \fract{e_0 r} (1 - \fract{e_0 r})
        \\&\quad + \sum_{j=1}^g \frac{l_j}{2n_j} \fract{e_{j-1} r}(1-\fract{e_{j-1} r})
        - \sum_{j=1}^g \sum_{b=1}^{n_j-1} \fract*{w_j\left(\fract{e_j r} - \frac{b}{n_j}\right)} \I{\left[\frac{b}{n_j},1\right)}\!\left(\fract{e_j r}\right),
\end{align*}
and we can rewrite the first two terms as follows:
\begin{align*}
    \frac{\mu_g - n_g w_g}{2} r + \frac{n_g w_g}{2} r^2
    &= \frac{\mu_g}{2} r^2 + \frac{n_g+w_g-1-n_g\mu_{g-1}}{2} r^2 +\frac{-n_g-w_g+1-n_g\mu_{g-1}}{2} r
    \\&= \frac{\mu_g}{2} r^2 - \frac{n_g+w_g-1-n_g\mu_{g-1}}{2} r(1-r).
\end{align*}
The numerator can be rewritten using lemmas \ref{lemma:wj} and \ref{lemma:ngmug}:
\begin{align*}
    n_g+w_g-1-n_g\mu_{g-1}
    &= n_g
        + \left(
            \sum_{k=1}^g l_k \frac{e_k^2 n_k}{n_g} + \frac{e_0^2}{n_g}
        \right)
        - 1
        - \left(
            \sum_{k=1}^{g-1}  l_k \left(\frac{e_k^2 n_k}{n_g} - \frac{e_k^2 n_k}{e_{k-1}}\right)
            + \frac{e_0^2}{n_g}
            + n_g
            - 2e_0
        \right)
    \\&= l_g
        + 2e_0
        - 1
        + \sum_{k=1}^{g-1}  l_k e_k
    = 2e_0
        - 1
        + \sum_{j=1}^g  l_j e_j.
\end{align*}
Combining these expressions we get
\begin{align*}
    2\mu\, \phi_f(r)
    &= \left(2e_0 - 1 + \sum_{j=1}^g  l_j e_j\right) r(1-r)
        - \sum_{j=1}^g (n_j -1) \fract{e_j r}
        - \fract{e_0 r} (1 - \fract{e_0 r})
        \\&\quad - \sum_{j=1}^g \frac{l_j}{n_j} \fract{e_{j-1} r}(1-\fract{e_{j-1} r})
        + \sum_{j=1}^g \sum_{b=1}^{n_j-1} 2 \fract*{w_j\left(\fract{e_j r} - \frac{b}{n_j}\right)} \I{\left[\frac{b}{n_j},1\right)}\!\left(\fract{e_j r}\right).
\end{align*}
This finishes the proof of Theorem \ref{theor:phi}.

\section{Characterization of the limit distribution for irreducible plane curves}\label{section:characterization}

In this section we will prove the main result of the paper which is the following:

\begin{theorem}\label{theor:characterization}
    Let $\left(f^{(i)}\right)_{i\geqslant 0}$ be a sequence of irreducible plane curves. Then, the distribution of Hodge spectral exponents of $f^{(i)}$ has limit
    \begin{equation*}
        \lim_{i\rightarrow +\infty} D_{f^{(i)}}(s) = N_2(s)
    \end{equation*}
    if, and only if,
    \begin{equation*}
        \lim_{i\rightarrow +\infty} \frac{\beta_g^{(i)}}{\mu^{(i)}} = 0,
    \end{equation*}
    where $\frac{\beta_g^{(i)}}{e_0^{(i)}}$
    is the last characteristic exponent of $f^{(i)}$ and $\mu^{(i)}$ is the Milnor number of $f^{(i)}$.
\end{theorem}
\begin{proof}
Let $f$ define an irreducible plane curve with Puiseux pairs $(n_1,l_1),\dots,(n_g,l_g)$.
We want to characterize all sequences of irreducible plane curves such that
$ \lim D_f(s) = N_2(s)$ or equivalently
$ \lim \phi_f(r) = 0$. Recall that Theorem \ref{theor:phi} states:

\begin{align*}
    \phi_f(r)
    &= \frac{1}{2\mu} \left( \biggl(2e_0 - 1 + \sum_{j=1}^g  l_j e_j\biggr) r(1-r)
        - \sum_{j=1}^g (n_j -1) \fract{e_j r}
        - \fract{e_0 r} (1 - \fract{e_0 r})
        \right.\\&\left.\hspace{15mm} - \sum_{j=1}^g \frac{l_j}{n_j} \fract{e_{j-1} r}(1-\fract{e_{j-1} r})
        + \sum_{j=1}^g \sum_{b=1}^{n_j-1} 2 \fract*{w_j\left(\fract{e_j r} - \frac{b}{n_j}\right)} \I{\left[\frac{b}{n_j},1\right)}\!\left(\fract{e_j r}\right)
        \right)
\end{align*}
with
\begin{equation*}
    \mu
    = \sum_{j=1}^g l_j e_j (e_{j-1}-1) + (e_0-1)^2.
\end{equation*}

\vskip 2mm

We will see in forthcoming Theorem \ref{theor:notLctImplies} that $\lim \phi_f(r) = 0$ implies $\lct(f) \rightarrow 0$, so in the following we will assume this to be the case. 
This is equivalent, by Corollary \ref{prop:lct}, to assuming $e_0\rightarrow +\infty$. By Lemma \ref{lemma:mu}, we have $ \frac{1}{\mu}  \xrightarrow{e_0\rightarrow +\infty} 0$ and we also have 

\begin{align*}
    \frac{1}{\mu} \left( \sum_{j=1}^g (n_j -1) \right)  \xrightarrow{e_0\rightarrow +\infty} 0    \; , \hskip 1cm 
    \frac{1}{\mu} \left( \sum_{j=1}^g \sum_{b=1}^{n_j-1} 1 \right)  \xrightarrow{e_0\rightarrow +\infty} 0
\end{align*}
and 
\begin{equation*}
    \frac{1}{\mu} \left( 2e_0 - 1 \right)
    = \frac{2e_0 - 1}{\sum_{j=1}^g l_j e_j (e_{j-1}-1) + (e_0-1)^2}
    \xrightarrow{e_0\rightarrow +\infty} 0.
\end{equation*}
These limits imply
\begin{align*}
    \lim \phi_f(r) = 0
    \iff
    e_0\rightarrow +\infty,\quad
    \frac{1}{\mu} \left(
        \biggl(\sum_{j=1}^g  l_j e_j\biggr) r(1-r)
        -
        \sum_{j=1}^g \frac{l_j}{n_j} \fract{e_{j-1} r}(1-\fract{e_{j-1} r})
    \right) \rightarrow 0
\end{align*}
since all functions of $r$ are bounded between 0 and 1 (excluding signs and coefficients). 
We want the limit of this expression to be 0 for $r\in[0,1)$. In particular it must be 0 for $r\in\left[\frac{1}{4},\frac{3}{4}\right]$ and, in this interval,
\begin{align*}
    &\sum_{j=1}^g \frac{l_j}{n_j} \fract{e_{j-1} r}(1-\fract{e_{j-1} r})
    \leqslant \sum_{j=1}^g \frac{l_j}{n_j} \frac{1}{4}
    \leqslant \sum_{j=1}^g l_j \frac{1}{8}
    \leqslant \frac{1}{8} \sum_{j=1}^g l_j e_j \,,
    \\&\left(\sum_{j=1}^g l_j e_j\right) r(1-r)
    \geqslant \frac{3}{16} \sum_{j=1}^g  l_j e_j
\end{align*}
and thus
\begin{align*}
    \frac{1}{\mu} \left(
        \biggl(\sum_{j=1}^g  l_j e_j\biggr) r(1-r)
        -
        \sum_{j=1}^g \frac{l_j}{n_j} \fract{e_{j-1} r}(1-\fract{e_{j-1} r})
    \right)
    \geqslant \frac{1}{\mu} \left(\frac{1}{16} \sum_{j=1}^g  l_j e_j\right).
\end{align*}
Therefore,
\begin{align*}
    \lim \phi_f(r) = 0
    \implies
    \frac{\sum_{j=1}^g  l_j e_j}{\mu} \rightarrow 0.
\end{align*}
Furthermore, if we assume $\frac{\sum_{j=1}^g  l_j e_j}{\mu} \rightarrow 0$, then for all $r\in[0,1)$ we have
\begin{align*}
    0 &\leqslant
    \frac{1}{\mu} \abs{
        \biggl(\sum_{j=1}^g  l_j e_j\biggr) r(1-r)
        -
        \sum_{j=1}^g \frac{l_j}{n_j} \fract{e_{j-1} r}(1-\fract{e_{j-1} r})
    }
    \\ & \leqslant
    \frac{1}{\mu} \max\left\{
        \sum_{j=1}^g  l_j e_j,\ 
        \sum_{j=1}^g \frac{l_j}{n_j}
    \right\}
    \leqslant
    \frac{1}{\mu} \sum_{j=1}^g  l_j e_j
    \rightarrow 0.
\end{align*}
In conclusion, we have proved
\begin{align*}
    \lim \phi_f(r) = 0
    \iff
    e_0\rightarrow +\infty,\quad
    \frac{\sum_{j=1}^g  l_j e_j}{\mu} \rightarrow 0.
\end{align*}

We can rewrite this expression using the definitions $  e_0 = n_g\cdots n_1
    ,\quad
    e_i = n_g\cdots n_{i+1}$ which imply $ \frac{e_0}{e_i} = n_i\cdots n_1$. Then we have
\begin{align*}
    \frac{\beta_g}{e_0} &= 1 + \frac{l_1}{n_1} + \dots + \frac{l_i}{n_1\cdots n_i} + \dots + \frac{l_g}{n_1\cdots n_g}
    = 1 + \frac{l_1}{\frac{e_0}{e_1}} + \dots + \frac{l_i}{\frac{e_0}{e_i}} + \dots + \frac{l_g}{\frac{e_0}{e_g}}
    \\&= \frac{1}{e_0} (e_0 + l_1 e_1 + \dots + l_i e_i + \dots + l_g e_g)
    = \frac{1}{e_0} \left(e_0 + \sum_{j=1}^g  l_j e_j\right) ,
\end{align*}
which implies  $\beta_g = e_0 + \sum_{j=1}^g  l_j e_j.$
Since we are in the case $e_0\rightarrow +\infty$ and $\frac{e_0}{\mu}\xrightarrow{e_0\rightarrow +\infty} 0$, the previous equivalence can be written as
\begin{align*}
    \lim \phi_f(r) = 0
    \iff
    e_0\rightarrow +\infty,\quad
    \frac{\beta_g}{\mu} \rightarrow 0.
\end{align*}

Now assume that we have a sequence $\left(f^{(i)}\right)_{i\geqslant 0}$ with
$e_0^{(i)}\not\rightarrow +\infty$. Then
\begin{align*}
    e_0^{(i)}\not\rightarrow +\infty
    &\implies \exists \;  \text{subsequence with $e_0^{(i)}$ constant}
    \\&\implies \exists \;  \text{subsequence $\left(f^{(i)}\right)_{i\in I}$ with $g^{(i)},e_0^{(i)},\dots,e_g^{(i)}$ constant}.
\end{align*}
For this subsequence,
\begin{align*}
    &\frac{\beta_g^{(i)}}{\mu^{(i)}}
    = \frac{e_0 + \sum_{j=1}^g  l_j^{(i)} e_j}{\sum_{j=1}^g l_j^{(i)} e_j (e_{j-1}-1) + (e_0-1)^2} \rightarrow 0
    \\\iff& \frac{\sum_{j=1}^g  l_j^{(i)} e_j}{\sum_{j=1}^g l_j^{(i)} e_j (e_{j-1}-1)} \rightarrow 0
    \\\iff& \frac{\sum_{j=1}^g l_j^{(i)} e_j (e_{j-1}-1)}{\sum_{j=1}^g  l_j^{(i)} e_j} \rightarrow +\infty
    \\\iff& \exists k\in\{1,\dots,g\}\quad \frac{l_k^{(i)} e_k (e_{k-1}-1)}{\sum_{j=1}^g  l_j^{(i)} e_j} \rightarrow +\infty
    \\\iff& \exists \; k\in\{1,\dots,g\}\quad \frac{\sum_{j=1}^g  l_j^{(i)} e_j}{l_k^{(i)}} \rightarrow 0
    \\\iff& \exists \; k\in\{1,\dots,g\}\colon\forall j\in\{1,\dots,g\}\quad \frac{l_j^{(i)} e_j}{l_k^{(i)}} \rightarrow 0 \,.
\end{align*}
But since $\frac{l_k^{(i)} e_k}{l_k^{(i)}} \not\rightarrow 0$, we can not have $\frac{\beta_g^{(i)}}{\mu^{(i)}} \rightarrow 0$ for the subsequence and thus it can not be true for the full sequence either. This proves
\begin{align*}
    e_0\not\rightarrow +\infty
    \implies
    \frac{\beta_g}{\mu} \not\rightarrow 0, \hskip 10mm {\rm and \; thus} \hskip 10mm  \frac{\beta_g}{\mu} \rightarrow 0
    \implies
    e_0\rightarrow +\infty.
\end{align*}
In conclusion,
\begin{align*}
    \lim \phi_f(r) = 0
    \iff
    \frac{\beta_g}{\mu} \rightarrow 0.
\end{align*}
\end{proof}

The following corollary is a generalization of the result of K.~Saito presented in Proposition \ref{prop:limitSaito}.

\begin{corollary}
    Let $f\in\C\{x,y\}$ be an irreducible plane curve with Puiseux pairs $(n_1,l_1),\dots,(n_g,l_g)$. Then, taking a sequence of such functions with the limit $n_k\rightarrow +\infty$ (keeping all other $n_j$ and $l_j$ fixed), one has
\begin{equation*}
    \lim_{i\rightarrow +\infty} D_{f^{(i)}}(s) = N_2(s).
\end{equation*}
\end{corollary}

\begin{proof}
Consider $n_k\rightarrow +\infty$, $g$ fixed, $n_j$ fixed for $j\neq k$ and $l_j$ fixed $\forall j$. Then,
\begin{align*}
    \beta_g = e_0 + \sum_{j=1}^g l_j e_j = O(1) n_k + \sum_{j=1}^{k-1} O(1) n_k + \sum_{j=k}^{g} O(1) = O(n_k),
\end{align*}
where we use the usual asymptotic \textit{big O notation}, and
\begin{align*}
    \mu
    = \sum_{j=1}^g l_j e_j (e_{j-1}-1) + (e_0-1)^2
   = \sum_{j=1}^{k-1} O(n_k^2) + O(n_k) + \sum_{j=k+1}^g O(1) + O(n_k^2)
    = O(n_k^2).
\end{align*}
Here we used $e_i=n_g\cdots n_{i+1}$ and thus the assymptotic behaviour when $n_k\rightarrow +\infty$ is
$O(1)n_k$ for $i < k$ and $O(1)$ otherwise. Therefore
\begin{align*}
    \frac{\beta_g}{\mu} = O\left(\frac{n_k}{n_k^2}\right)
    \implies
    \frac{\beta_g}{\mu} \rightarrow 0.
\end{align*}
This satisfies the hypothesis for Theorem \ref{theor:characterization}, so
we get the desired result.
\end{proof}

\section{On the condition \texorpdfstring{$\lct \rightarrow 0$}{lct→0}} \label{LCT}

The examples provided by K.~Saito \cite{Ksaito1983zeroes} and Almir\'on and Schulze \cite{almiron2022limit} satisfy that the log-canonical thresholds of the sequence of hypersurfaces tend to zero. 
In this section we will prove that this is a necessary condition but it is not sufficient for the discrete distribution converging to the continuous distribution.

\begin{theorem}\label{theor:notLctImplies}
    Let $\left(f^{(i)}\right)_{i\geqslant 0}$ be a sequence of hypersurfaces of $\mathbb{C}^{n+1}$ with an isolated singular point at the origin. Then
    \begin{align*}
        \lct \left(f^{(i)}\right)\not\rightarrow 0 \implies
        \lim_{i\rightarrow +\infty} D_{f^{(i)}}(s) \neq N_{n+1}(s).
    \end{align*}
\end{theorem}

\begin{proof}
    We consider a sequence of hypersurfaces $\left(f^{(i)}\right)_{i\geqslant 0}$ with $\lct\left(f^{(i)}\right)\not\rightarrow 0$. Then, there exists $\varepsilon > 0$ and a subsequence $\left(f^{(i_j)}\right)_{j\geqslant 0}$ for some $(i_j)_{j\geqslant 0}\subseteq\N$ such that $\varepsilon \leqslant \lct\left(f^{(i_j)}\right)\ \forall j$. Thus,
    \begin{equation*}
        0 < \varepsilon \leqslant \lct\left(f^{(i_j)}\right) \leqslant \alpha_k^{(i_j)}\quad \forall j,k,
    \end{equation*}
    so there are no Hodge spectral exponents of any $f^{(i_j)}$ in the interval $(0,\varepsilon)$. In terms of the distribution,
    \begin{equation*}
        \lim_{j\rightarrow +\infty} D_{f^{(i_j)}}(s) = 0 \quad\text{for } s\in(0,\varepsilon).
    \end{equation*}
    However, the usual limit distribution satisfies $N_{n+1}(s) > 0 \quad\text{for } s\in(0,\varepsilon).$
    Therefore,
    \begin{equation*}
        \lim_{j\rightarrow +\infty} D_{f^{(i_j)}}(s) \neq N_{n+1}(s)
    \end{equation*}
    and thus
    \begin{equation*}
        \lim_{i\rightarrow +\infty} D_{f^{(i)}}(s) \neq N_{n+1}(s).
    \end{equation*}
or equivalently, 
        \begin{equation*}
            \lim_{i\rightarrow +\infty} \Chi_{f^{(i)}}(T) \neq \left(\frac{T-1}{\log{T}}\right)^{n+1} = \Fourier{N_{n+1}(s)}(\tau).
        \end{equation*}
\end{proof}

We will also prove
\begin{equation*}
    \lct\!\left(f^{(i)}\right)\rightarrow 0 \centernot\implies \lim_{i\rightarrow +\infty} D_{f^{(i)}}(s) = N_{n+1}(s).
\end{equation*}
by means of the following counterexample:
\begin{lemma}\label{lemma:counter}
    Consider a sequence of irreducible plane curves with Puiseux pairs $(2k,1), (2,2k^3+1)$ with $k\in\Z_{>0}$. Then
    \begin{equation*}
        \lct(f) = \frac{1}{2}\left(\frac{1}{2k}+\frac{1}{2k+1}\right) \rightarrow 0.
    \end{equation*}
    However, the limit of the distribution of Hodge spectral exponents is
    \begin{align*}
        \lim_{k\rightarrow +\infty} D_f(s)
        &= \I{\left[\frac{1}{2},\frac{3}{2}\right]}(s)
        \neq N_2(s).
    \end{align*}
\end{lemma}

\begin{proof}
 Consider a sequence of irreducible plane curves with Puiseux pairs $(2k,1),\, (2,2k^3+1)$ for $k\in\Z_{>0}$ ($k\rightarrow +\infty$). We have the following numerical invariants:
\begin{align*}
    e_2 = 1
    ,\qquad
    e_1 = 2
    ,\qquad
    e_0 = 4k
    ,
\end{align*}
and from Lemma \ref{lemma:mu}
\begin{align*}
    \mu
    &= l_1 e_1 (e_0-1) + l_2 e_2 (e_1-1) + (e_0-1)^2
    = 2k^3 + 16k^2.
\end{align*}
Now, recall the expression of $\phi_f(r)$ from Theorem \ref{theor:phi}:
\begin{align*}
    \phi_f(r)
    &= \frac{1}{2\mu} \left(
        \biggl(2e_0 - 1 + \sum_{j=1}^g  l_j e_j\biggr) r(1-r)
        - \sum_{j=1}^g (n_j -1) \fract{e_j r}
        - \fract{e_0 r} (1 - \fract{e_0 r})
        \right.\\&\left.\hspace{15mm}
        - \sum_{j=1}^g \frac{l_j}{n_j} \fract{e_{j-1} r}(1-\fract{e_{j-1} r})
        + \sum_{j=1}^g \sum_{b=1}^{n_j-1} 2 \fract*{w_j\left(\fract{e_j r} - \frac{b}{n_j}\right)} \I{\left[\frac{b}{n_j},1\right)}\!\left(\fract{e_j r}\right)
    \right).
\end{align*}
%
Then:
\begin{enumerate}[label={\arabic*)}]
    \item The limit of the first term is $r(1-r)$ times
        \begin{align*}
            \frac{1}{2\mu} \biggl(2e_0 - 1 + \sum_{j=1}^g  l_j e_j\biggr)
            &= \frac{1}{2\mu} \biggl(8k - 1 + 2 + (2k^3+1)\biggr)
            = \frac{2k^3 + 8k + 2}{4k^3 + 32k^2}
            \xrightarrow{k\rightarrow +\infty} \frac{1}{2}.
        \end{align*}
    \item The limit of the second term is
        \begin{align*}
            \frac{1}{2\mu} \abs{- \sum_{j=1}^g (n_j -1) \fract{e_j r}}
            &\leqslant \frac{1}{2\mu} \abs{\sum_{j=1}^g (n_j -1)}
            = \frac{\abs{2k-1+2-1}}{4k^3 + 32k^2}
            \xrightarrow{k\rightarrow +\infty} 0.
        \end{align*}
    \item The limit of the third term is
        \begin{align*}
            \frac{\abs{- \fract{e_0 r} (1 - \fract{e_0 r})}}{2\mu}
            &\leqslant \frac{\abs{1}}{2\mu}
            = \frac{1}{4k^3 + 32k^2}
            \xrightarrow{k\rightarrow +\infty} 0.
        \end{align*}
        %
    \item For the limit of the fourth term 
    we separate each summand:
        \begin{itemize}
            \item For $j=1$:
                \begin{align*}
                    \frac{1}{2\mu} \abs{\frac{l_1}{n_1} \fract{e_0 r}(1-\fract{e_0 r}) }
                    &\leqslant \frac{1}{2\mu} \abs{\frac{l_1}{n_1}}
                    = \frac{1}{2\mu} \abs{\frac{1}{2k}}
                    \xrightarrow{k\rightarrow +\infty} 0.
                \end{align*}
            \item For $j=2$ the coefficient is:
                \begin{align*}
                    \frac{1}{2\mu} \left( \frac{l_2}{n_2} \right)
                    &= \frac{1}{2\mu} \left( \frac{2k^3+1}{2} \right)
                    = \frac{2k^3+1}{8k^3 + 64k^2}
                    \xrightarrow{k\rightarrow +\infty} \frac{1}{4}.
                \end{align*}
        \end{itemize}
        Thus the fourth term has limit
        \begin{align*}
            \frac{1}{2\mu} \left( \sum_{j=1}^g \frac{l_j}{n_j} \fract{e_{j-1} r}(1-\fract{e_{j-1} r}) \right)
            \xrightarrow{k\rightarrow +\infty} \frac{1}{4} \fract{e_1 r}(1-\fract{e_1 r})
            = \frac{1}{4} \fract{2 r}(1-\fract{2 r}).
        \end{align*}
    \item The limit of the fifth term is
        \begin{align*}
            &\frac{1}{2\mu} \abs{\sum_{j=1}^g \sum_{b=1}^{n_j-1} 2 \fract*{w_j\left(\fract{e_j r} - \frac{b}{n_j}\right)} \I{\left[\frac{b}{n_j},1\right)}\!\left(\fract{e_j r}\right)}
            \leqslant \frac{1}{2\mu} \abs{\sum_{j=1}^g \sum_{b=1}^{n_j-1} 2}
            = \frac{1}{\mu} \abs{\sum_{j=1}^g (n_j -1)}
            \\&= \frac{\abs{2k-1+2-1}}{2k^3 + 16k^2}
            \xrightarrow{k\rightarrow +\infty} 0.
        \end{align*}
\end{enumerate}

\noindent
In total, for $r\in[0,1)$,
\begin{align*}
    \lim_{k\rightarrow +\infty} \phi_f(r)
    &= \frac{1}{2} r(1-r) - \frac{1}{4} \fract{2 r}(1-\fract{2 r}).
\end{align*}
By the definition of $\phi_f(r)$,
\begin{align*}
    \phi_f(r) &= \int_0^r N_2(s) - D_f^{<1}(s) \dd{s},
    \\
    \lim_{k\rightarrow +\infty} \phi_f(r) &= \int_0^r N_2(s) - \left( \lim_{k\rightarrow +\infty} D_f^{<1}(s)\right) \dd{s}.
\end{align*}
Since in this case $\lim_{k\rightarrow +\infty} \phi_f(r)$ is a continuous function and differentiable almost everywhere, we can recover the corresponding distribution as follows. Let $s\in[0,1)$, then
\begin{align*}
    \dv{}{s} \left(\lim_{k\rightarrow +\infty} \phi_f(s)\right) &= N_2(s) - \left(\lim_{k\rightarrow +\infty} D_f^{<1}(s)\right),
    \\
    \lim_{k\rightarrow +\infty} D_f^{<1}(s)
    &= N_2(s) - \dv{}{s} \left(\lim_{k\rightarrow +\infty} \phi_f(s)\right)
    \\&= s - \dv{}{s} \left( \frac{1}{2} s(1-s) - \frac{1}{4} \fract{2 s}(1-\fract{2 s}) \right)
    \\&= s - \left( \frac{1}{2} (1-2s) - \frac{1}{4} (2-4\fract{2 s}) \right)
    \\&= 2s - \fract{2 s}
    = \floor{2s}
    = \I{\left[\frac{1}{2},1\right)}(s).
\end{align*}
By the symmetry of the Hodge spectral exponents, for any $s$
\begin{equation*}
    \lim_{k\rightarrow +\infty} D_f(s)
    = \I{\left[\frac{1}{2},\frac{3}{2}\right]}(s).
\end{equation*}
\end{proof}

\section{Dominating values for irreducible plane curves}\label{section:dominating}
In this section we will give partial answers to K.~Saito's Question \ref{Q2} on the dominating values for the case of irreducible plane curves. To such purpose we will use Theorem \ref{theor:phi} which states
\begin{align*}
    2\mu\, \phi_f(r)
    &= \left(2e_0 - 1 + \sum_{j=1}^g  l_j e_j\right) r(1-r)
        - \sum_{j=1}^g (n_j -1) \fract{e_j r}
        - \fract{e_0 r} (1 - \fract{e_0 r})
        \\&\quad - \sum_{j=1}^g \frac{l_j}{n_j} \fract{e_{j-1} r}(1-\fract{e_{j-1} r})
        + \sum_{j=1}^g \sum_{b=1}^{n_j-1} 2 \fract*{w_j\left(\fract{e_j r} - \frac{b}{n_j}\right)} \I{\left[\frac{b}{n_j},1\right)}\!\left(\fract{e_j r}\right).
\end{align*}

We will give two lower bounds to this function and calculate the sets where the bound is positive. This way we obtain two intervals where $\phi_f$ is positive.
The following picture illustrates the idea.

\begin{figure}[htbp!]
   \centering
   \includegraphics[scale=0.4]{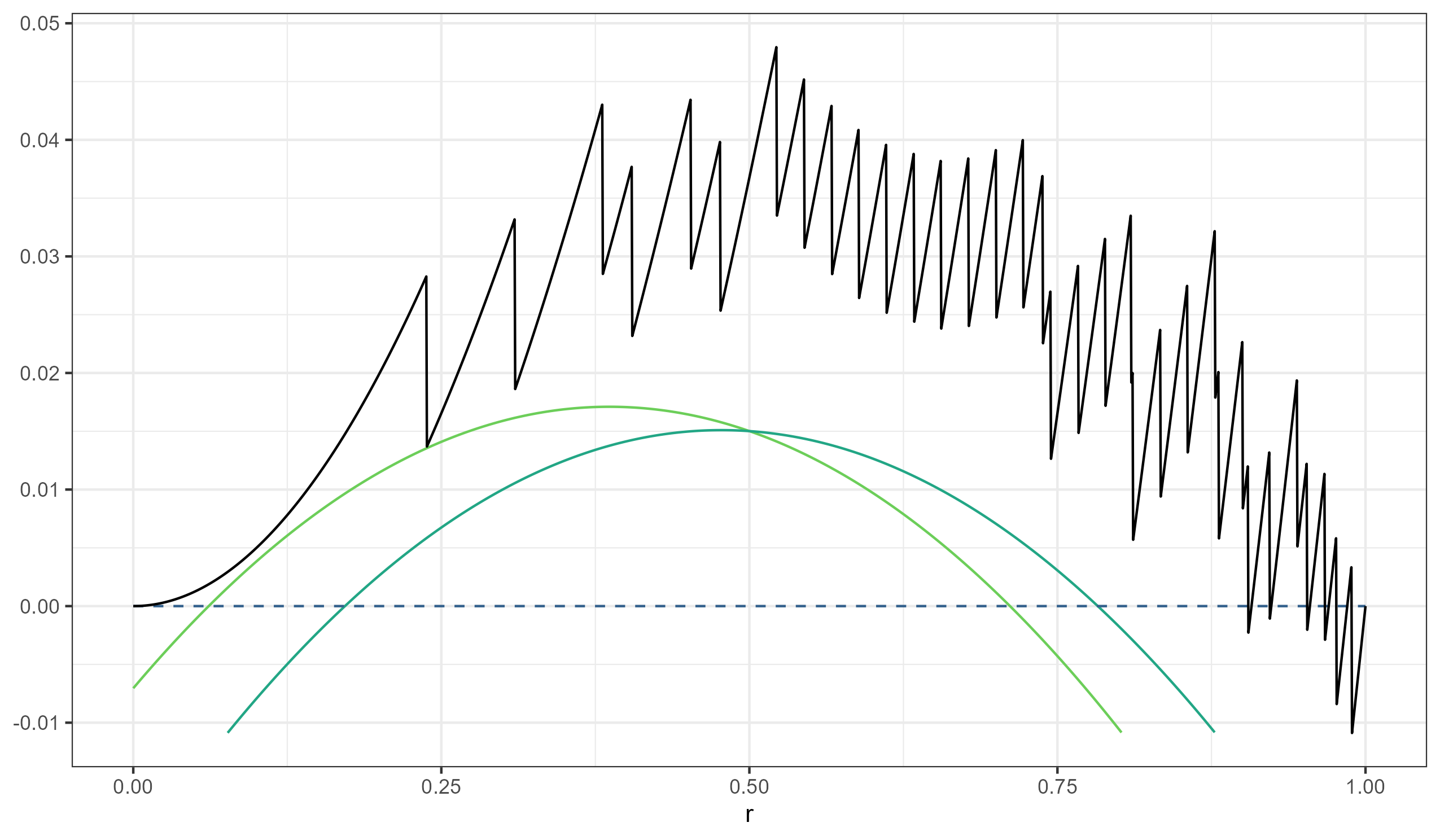}
   \caption{Example of the cumulative difference function $\phi_f(r)$ (black line), the first lower bound (dark green parabola) and the second lower bound (light green parabola), with the value 0 (dashed line) as a reference. The intervals where the parabolas are positive give us sets of dominating values. 
   This plot corresponds to an irreducible plane curve with Puiseux pairs $(3,4), (2,3)$.}
\end{figure}

\vskip 2mm 
We point out that both intervals always intersect but it is not always clear which are the ends of the unique interval of dominating values they provide.
In the main theorem that we anounce next, we also present two other intervals, at both ends of $[0,1)$, where $\phi_f(r)$ is positive and negative respectively. 
These intervals are obtained directly from the smallest and largest Hodge spectral exponents. 

\begin{theorem}\label{theor:intervals}
    Let $f$ be an irreducible plane curve with Puiseux pairs $(n_1,l_1),\dots,(n_g,l_g)$. Then,
    \begin{enumerate}[label={\arabic*)}]
        \item A set of dominating values is given by the interval
            \begin{equation*}
                r\in \left(
                \frac{
                    \left(2e_0 - n_g + \sum_{j=1}^g  l_j e_j\right) - \sqrt{D_1}
                }{
                    2 \left(2e_0 - 1 + \sum_{j=1}^g  l_j e_j\right)
                }
                ,\ 
                \frac{
                    \left(2e_0 - n_g + \sum_{j=1}^g  l_j e_j\right) + \sqrt{D_1}
                }{
                    2 \left(2e_0 - 1 + \sum_{j=1}^g  l_j e_j\right)
                }
                \right)
            \end{equation*}
            with
            \begin{equation*}
                D_1
                = \left(2e_0 - n_g + \sum_{j=1}^g  l_j e_j\right)^2
                    - 4
                    \left(2e_0 - 1 + \sum_{j=1}^g  l_j e_j\right)
                    \left( \sum_{j=1}^{g-1} (n_j -1) + \frac{1}{4} + \sum_{j=1}^g \frac{l_j}{4n_j} \right)
                    >0.
            \end{equation*}
        \item A set of dominating values is given by the interval
            \begin{equation*}
                r\in \left(
                \frac{
                    \left(e_0 + \sum_{j=1}^g  l_j e_j\right) - \sqrt{D_2}
                }{
                    2 \left(2e_0 - 1 + \sum_{j=1}^g  l_j e_j\right)
                }
                ,\ 
                \frac{
                    \left(e_0 + \sum_{j=1}^g  l_j e_j\right) + \sqrt{D_2}
                }{
                    2 \left(2e_0 - 1 + \sum_{j=1}^g  l_j e_j\right)
                }
                \right)
            \end{equation*}
            with
            \begin{equation*}
                D_2
                = \left(e_0 + \sum_{j=1}^g  l_j e_j\right)^2
                    - 4
                    \left(2e_0 - 1 + \sum_{j=1}^g  l_j e_j\right)
                    \left(\frac{1}{4} + \sum_{j=1}^g \frac{l_j}{4n_j} \right)
                    >0.
            \end{equation*}
        \item We have that the leftmost interval of $(0,1)$
            \begin{align*}
                r\in \left(0, \lct(f) \right) = \left(0,\ \frac{1}{e_1} \left(\frac{1}{n_1}+\frac{1}{n_1+l_1}\right)\right)
            \end{align*}
            is a set of dominating values, whereas the rightmost interval
            \begin{align*}
                r\in \left[1 - \frac{1}{n_g w_g}, 1\right)
            \end{align*}
            satisfies $\phi_f(r)<0$.
    \end{enumerate}
\end{theorem}

\begin{remark}
    Almir\'on and Schulze \cite[Prop. 6]{almiron2022limit} proved that the log-canonical threshold of an irreducible plane curve is a dominating value except for the cases where the curve has semigroup $(2,3)$ or $(2,5)$.
\end{remark}

An immediate consequence is the following:

\begin{corollary}
 For any $r$ in the intervals described in parts $1)$ and $2)$ of Theorem \ref{theor:intervals}  we have
  \begin{equation*}
    r^2 > \frac{2}{\mu} \#\left\{\alpha_i\leqslant r\right\} .
\end{equation*}
\end{corollary}

In the course of the proof of Theorem \ref{theor:intervals} we will obtain an alternative proof of the result Tomari \cite{tomari1993inequality} but only for irreducible curves.

\begin{corollary} \label{tomari}
 Let $f$ define an irreducible plane curve. Then $\frac{1}{2}$ is a dominating value and thus
  \begin{equation*}
    \#\left\{\alpha_i\leqslant \frac{1}{2}\right\} < \frac{\mu}{8}.
\end{equation*}
\end{corollary}


\subsubsection{Interval of dominating values using the first lower bound of \texorpdfstring{$\phi_f$}{ϕ\_f}}

We will now give a lower bound of $\phi_f(r)$, from which we will obtain an interval of dominating values. We use the following bounds:   $\fract{e_g r} = r$, \;  $ \fract{e_j r} \leqslant 1$, \;  $  \fract{e_j r} (1-\fract{e_j r}) \leqslant \frac{1}{4} $ \; and $$    \fract*{w_j\left(\fract{e_j r} - \frac{b}{n_j}\right)} \I{\left[\frac{b}{n_j},1\right)}\!\left(\fract{e_j r}\right) \geqslant 0 .$$
%
Applying these bounds to the previous expression of $\phi_f(r)$ we get
\begin{align*}
    2\mu\, \phi_f(r)
    &= \left(2e_0 - 1 + \sum_{j=1}^g  l_j e_j\right) r(1-r)
        - \sum_{j=1}^g (n_j -1) \fract{e_j r}
        - \fract{e_0 r} (1 - \fract{e_0 r})
        \\&\qquad - \sum_{j=1}^g \frac{l_j}{n_j} \fract{e_{j-1} r}(1-\fract{e_{j-1} r})
        + \sum_{j=1}^g \sum_{b=1}^{n_j-1} 2 \fract*{w_j\left(\fract{e_j r} - \frac{b}{n_j}\right)} \I{\left[\frac{b}{n_j},1\right)}\!\left(\fract{e_j r}\right)
    \\&\geqslant \left(2e_0 - 1 + \sum_{j=1}^g  l_j e_j\right) r(1-r)
        - (n_g -1) r
        - \sum_{j=1}^{g-1} (n_j -1)
        - \frac{1}{4}
        - \sum_{j=1}^g \frac{l_j}{n_j} \frac{1}{4}
    \\&= -\left(2e_0 - 1 + \sum_{j=1}^g  l_j e_j\right) r^2
        + \left(2e_0 - n_g + \sum_{j=1}^g  l_j e_j\right) r
        - \left( \sum_{j=1}^{g-1} (n_j -1) + \frac{1}{4} + \sum_{j=1}^g \frac{l_j}{4n_j} \right)
    \\&\eqqcolon p_1(r).
\end{align*}
Therefore, $  2\mu\, \phi_f(r) \geqslant p_1(r).$
We will now find the interval where $p_1(r)>0$:

\vskip 2mm
\noindent Evaluating $p_1(r)$ at $r=0$,
\begin{align*}
    p_1(0)
    = - \left( \sum_{j=1}^{g-1} (n_j -1)
    + \frac{1}{4}
    + \sum_{j=1}^g \frac{l_j}{4n_j} \right)
    < - \left(0 + \frac{1}{4} + 0\right)
    < 0.
\end{align*}
Evaluating at $r=1$,
\begin{align*}
    p_1(1)
    &= -\left(2e_0 - 1 + \sum_{j=1}^g  l_j e_j\right)
        + \left(2e_0 - n_g + \sum_{j=1}^g  l_j e_j\right)
        - \left( \sum_{j=1}^{g-1} (n_j -1) + \frac{1}{4} + \sum_{j=1}^g \frac{l_j}{4n_j} \right)
    \\&= (1 - n_g)
        - \left( \sum_{j=1}^{g-1} (n_j -1)
        + \frac{1}{4}
        + \sum_{j=1}^g \frac{l_j}{4n_j} \right)
   < 0.
\end{align*}
Evaluating at $r=\frac{1}{2}$,
\begin{align*}
    p_1\!\left(\frac{1}{2}\right)
    &= -\left(2e_0 - 1 + \sum_{j=1}^g  l_j e_j\right) \frac{1}{4}
        + \left(2e_0 - n_g + \sum_{j=1}^g  l_j e_j\right) \frac{1}{2}
        - \left( \sum_{j=1}^{g-1} (n_j -1) + \frac{1}{4} + \sum_{j=1}^g \frac{l_j}{4n_j} \right)
    \\&= \frac{1}{4} \left(-2e_0 + 1 - \sum_{j=1}^g  l_j e_j + 4e_0 - 2n_g + 2\sum_{j=1}^g  l_j e_j
        -\sum_{j=1}^{g-1} 4(n_j -1) - 1 - \sum_{j=1}^g \frac{l_j}{n_j} \right)
    \\&= \frac{1}{4} \left(2e_0 - 2n_g 
        -\sum_{j=1}^{g-1} 4(n_j -1) + \sum_{j=1}^g  l_j \left(e_j - \frac{1}{n_j} \right) \right)
    \\&> \frac{1}{2} \left(e_0 - n_g - \sum_{j=1}^{g-1} 2n_j + 2(g-1) \right)
    = \frac{1}{2} \left(n_g\cdots n_1 - n_g - 2n_{g-1} - \dots - 2n_1 + 2(g-1) \right).
\end{align*}
This last term is 
greater or equal than zero since:
\begin{itemize}
    \item If $g>1$, then
        \begin{align*}
            &n_g\cdots n_1 - n_g - 2n_{g-1} - \dots - 2n_1 + 2(g-1)
            \\&\geqslant (n_g-2)n_{g-1}\cdots n_1 - n_g +2n_{g-1}\cdots n_1 - 2n_{g-1} - \dots - 2n_1 + 2
            \\&\geqslant (n_g-2)\cdot 2 - n_g + 2 (n_{g-1}\cdots n_1 - n_{g-1} - \dots - n_1) + 2
            \\&\geqslant 2n_g - 4 - n_g + 0 + 2
          = n_g - 2
           \geqslant 0.
        \end{align*}
    \item If $g=1$, then
        \begin{align*}
            &n_g\cdots n_1 - n_g - 2n_{g-1} - \dots - 2n_1 + 2(g-1)
           = n_1 - n_1 + 0
            = 0.
        \end{align*}
\end{itemize}
Therefore,
$
    p_1\!\left(\frac{1}{2}\right) > 0.
$
\begin{remark}
    This proves Corollary \ref{tomari}.
\end{remark}

Since $p_1(r)$ is a quadratic function, $p_1(0)<0$, $p_1\!\left(\frac{1}{2}\right) > 0$ and $p_1(1)<0$, we conclude that the graph of $p_1(r)$ is an inverted parabola with a root in $\left(0,\frac{1}{2}\right)$ and a root in $\left(\frac{1}{2},1\right)$. More concretely, $p_1(r)>0$ and therefore $\phi_f(r)>0$ in the interval
\begin{equation*}
    r\in \left(
    \frac{
    \left(2e_0 - n_g + \sum_{j=1}^g  l_j e_j\right) - \sqrt{D_1}
    }{
    2 \left(2e_0 - 1 + \sum_{j=1}^g  l_j e_j\right)
    }
    ,
    \frac{
    \left(2e_0 - n_g + \sum_{j=1}^g  l_j e_j\right) + \sqrt{D_1}
    }{
    2 \left(2e_0 - 1 + \sum_{j=1}^g  l_j e_j\right)
    }
    \right)
\end{equation*}
with
\begin{equation*}
    D_1 = \left(2e_0 - n_g + \sum_{j=1}^g  l_j e_j\right)^2
        - 4
        \left(2e_0 - 1 + \sum_{j=1}^g  l_j e_j\right)
        \left( \sum_{j=1}^{g-1} (n_j -1) + \frac{1}{4} + \sum_{j=1}^g \frac{l_j}{4n_j} \right)
        >0.
\end{equation*}

\subsubsection{Interval of dominating values using the second lower bound of \texorpdfstring{$\phi_f$}{ϕ\_f}}

We will now give a different lower bound of $\phi_f(r)$, from which we will obtain another interval of dominating values.
If instead of bounding $\fract{e_j r} \leqslant 1$ we bound it by $\fract{e_j r} \leqslant e_j r$, we will get a better bound for small values of $r$ (at least for $r < \frac{1}{e_1}$). Thus, now we use the following bounds:
\begin{align*}
    \fract{e_j r} &\leqslant e_j r ,\\
    \fract{e_j r} (1-\fract{e_j r}) &\leqslant \frac{1}{4} ,\\
    \fract*{w_j\left(\fract{e_j r} - \frac{b}{n_j}\right)} \I{\left[\frac{b}{n_j},1\right)}\!\left(\fract{e_j r}\right) &\geqslant 0.
\end{align*}
Then we get
\begin{align*}
    2\mu\, \phi_f(r)
    &= \left(2e_0 - 1 + \sum_{j=1}^g  l_j e_j\right) r(1-r)
        - \sum_{j=1}^g (n_j -1) \fract{e_j r}
        - \fract{e_0 r} (1 - \fract{e_0 r})
        \\&\qquad - \sum_{j=1}^g \frac{l_j}{n_j} \fract{e_{j-1} r}(1-\fract{e_{j-1} r})
        + \sum_{j=1}^g \sum_{b=1}^{n_j-1} 2 \fract*{w_j\left(\fract{e_j r} - \frac{b}{n_j}\right)} \I{\left[\frac{b}{n_j},1\right)}\!\left(\fract{e_j r}\right)
    \\&\geqslant \left(2e_0 - 1 + \sum_{j=1}^g  l_j e_j\right) r(1-r)
        - \sum_{j=1}^g (n_j -1) e_j r
        - \frac{1}{4}
        - \sum_{j=1}^g \frac{l_j}{n_j} \frac{1}{4}
    \\&= -\left(2e_0 - 1 + \sum_{j=1}^g  l_j e_j\right) r^2
        + \left(2e_0 - 1 + \sum_{j=1}^g  l_j e_j - \sum_{j=1}^g (n_j -1) e_j\right) r
        - \left(\frac{1}{4} + \sum_{j=1}^g \frac{l_j}{4n_j} \right)
    \\&= -\left(2e_0 - 1 + \sum_{j=1}^g  l_j e_j\right) r^2
        + \left(e_0 + \sum_{j=1}^g  l_j e_j\right) r
        - \left(\frac{1}{4} + \sum_{j=1}^g \frac{l_j}{4n_j} \right)
    \\&\eqqcolon p_2(r).
\end{align*}
Therefore, $    2\mu\, \phi_f(r) \geqslant p_2(r).$
We will now find the interval where $p_2(r)>0$:


\noindent
Evaluating $p_2(r)$ at $r=0$,
\begin{align*}
    p_2(0)
    &= - \left(\frac{1}{4} + \sum_{j=1}^g \frac{l_j}{4n_j} \right)
    < 0.
\end{align*}
Evaluating at $r=1$
\begin{align*}
    p_2(1)
    &= -\left(2e_0 - 1 + \sum_{j=1}^g  l_j e_j\right)
        + \left(e_0 + \sum_{j=1}^g  l_j e_j\right)
        - \left(\frac{1}{4} + \sum_{j=1}^g \frac{l_j}{4n_j} \right)
    \\&= -(e_0 - 1) - \left(\frac{1}{4} + \sum_{j=1}^g \frac{l_j}{4n_j} \right)
    < 0.
\end{align*}
Evaluating at $r=\frac{1}{2}$
\begin{align*}
    p_2\!\left(\frac{1}{2}\right)
    &= -\left(2e_0 - 1 + \sum_{j=1}^g  l_j e_j\right) \frac{1}{4}
        + \left(e_0 + \sum_{j=1}^g  l_j e_j\right) \frac{1}{2}
        - \left(\frac{1}{4} + \sum_{j=1}^g \frac{l_j}{4n_j} \right)
    \\&= \frac{1}{4} \left(
        -2e_0 + 1 - \sum_{j=1}^g  l_j e_j
        + 2e_0 + 2\sum_{j=1}^g  l_j e_j
        - 1 - \sum_{j=1}^g \frac{l_j}{n_j}
    \right)
    \\&= \frac{1}{4} \left(
        \sum_{j=1}^g  l_j \left(e_j - \frac{1}{n_j}\right)
    \right)
    > 0.
\end{align*}
\begin{remark}
    This also proves Corollary \ref{tomari}.
\end{remark}

Since $p_2(r)$ is a quadratic function, $p_2(0)<0$, $p_2\!\left(\frac{1}{2}\right) > 0$ and $p_2(1)<0$, we conclude that the graph of $p_2(r)$ is an inverted parabola with a root in $\left(0,\frac{1}{2}\right)$ and a root in $\left(\frac{1}{2},1\right)$. More concretely, $p_2(r)>0$ and therefore $\phi_f(r)>0$ in the interval
\begin{equation*}
    r\in \left(
    \frac{
    \left(e_0 + \sum_{j=1}^g  l_j e_j\right) - \sqrt{D_2}
    }{
    2 \left(2e_0 - 1 + \sum_{j=1}^g  l_j e_j\right)
    }
    ,
    \frac{
    \left(e_0 + \sum_{j=1}^g  l_j e_j\right) + \sqrt{D_2}
    }{
    2 \left(2e_0 - 1 + \sum_{j=1}^g  l_j e_j\right)
    }
    \right)
\end{equation*}
with
\begin{equation*}
    D_2 = \left(e_0 + \sum_{j=1}^g  l_j e_j\right)^2
        - 4
        \left(2e_0 - 1 + \sum_{j=1}^g  l_j e_j\right)
        \left(\frac{1}{4} + \sum_{j=1}^g \frac{l_j}{4n_j} \right)
        >0.
\end{equation*}

\subsubsection{Intervals of constant sign at both ends of \texorpdfstring{$[0,1)$}{[0,1)}}

 These intervals come directly from the first and the last Hodge spectral exponents 
in the set given in Corollary \ref{prop:lct} and will not use the explicit expression for $\phi_f(r)$. 
For the leftmost interval, the log-canonical threshold is
\begin{equation*}
    \alpha_1=\lct(f)
    = \frac{1}{e_1} \left(\frac{1}{n_1}+\frac{1}{n_1+l_1}\right).
\end{equation*}
 Thus, $\#\{\alpha_i < \alpha_1\}=0$ and $\phi_f(r)>0$ in the interval $r\in \bigl( 0, \alpha_1 \bigr)$.
For the rightmost interval, the largest  Hodge spectral exponent is 
\begin{align*}
  \alpha_{\mu/2}=   1 - \frac{1}{e_{g-1} w_g}
    = 1 - \frac{1}{n_g w_g}
\end{align*}
and thus we have $\phi_f(r)<0$ in the interval
$
   r\in \left[1 - \alpha_{\mu/2}, 1\right).
$


\nocite{*}
\bibliographystyle{amsalpha} 
\providecommand{\bysame}{\leavevmode\hbox to3em{\hrulefill}\thinspace}
\providecommand{\MR}{\relax\ifhmode\unskip\space\fi MR }
\providecommand{\MRhref}[2]{%
  \href{http://www.ams.org/mathscinet-getitem?mr=#1}{#2}
}
\providecommand{\href}[2]{#2}


\end{document}